\documentclass[12pt,fleqn]{article}

\usepackage{amsmath}
\usepackage{amsthm}
\usepackage{amssymb}
\usepackage{amsfonts}
\usepackage{epsf}
\usepackage{graphicx}
\usepackage{hyperref}
\usepackage{color}

\newtheorem{lemma}{Lemma}

\newtheorem{theorem}{Theorem}

\newtheorem{prop}{Proposition}

\theoremstyle{remark}
\newtheorem{rmk}{Remark}
\theoremstyle{definition}

\DeclareMathOperator*\Res{{Res}} \DeclareMathOperator\Ai{{Ai}}
 \DeclareMathOperator\re{{Re}}
\DeclareMathOperator\im{{Im}} \numberwithin{equation}{section}

\newcommand{\D}{\displaystyle}

\numberwithin{equation}{section}

\newcounter{comment}
\def\a{\alpha}

\def\l{\lambda}

\def\cut{\setminus}

\begin{document}

\title{Connection formulas for the Ablowitz-Segur solutions of the inhomogeneous Painlev\'e II equation}
\date{\today}
\author{Dan Dai$^{\ast}$ and Weiying Hu$^{\dag}$}

\maketitle

\begin{abstract}
   We consider the second Painlev\'e equation
   $$
   u''(x)=2u^3(x)+xu(x)-\alpha,
   $$
   where $\alpha $ is a nonzero constant. Using the Deift-Zhou nonlinear steepest descent method for Riemann-Hilbert problems, we rigorously prove the asymptotics as $x \to \pm \infty$ for both the real and purely imaginary Ablowitz-Segur solutions, as well as the corresponding connection formulas. We also show that the real Ablowitz-Segur solutions have no real poles when $\alpha \in (-1/2, 1/2)$.
\end{abstract}


\vspace{2cm}




\vspace{5mm}

\hrule width 65mm

\vspace{2mm}

\begin{description}

\item \hspace*{5mm}$\ast$ Department of Mathematics, City University of
Hong Kong, Hong Kong. \\
Email: \texttt{dandai@cityu.edu.hk}

\item \hspace*{5mm}$\dag$ Department of Mathematics, City University of
Hong Kong, Hong Kong. \\
Email: \texttt{weiyinghu2-c@my.cityu.edu.hk} (corresponding author)

\end{description}

\newpage

\section{Introduction and statement of results}

The second Painlev\'e equation (PII) is the following second-order nonlinear differential equation
\begin{equation}\label{PII-def}
   u''(x)=2u^3+xu-\a,
\end{equation}
where $\alpha$ is a constant. When $\alpha = 0$, this equation is called \emph{homogeneous PII}, otherwise it is called \emph{inhomogeneous PII}. The PII equation, together with the other five ones, was first studied by Painlev\'e and his colleagues at the turn of twentieth century; see more information about the Painlev\'e equation and the historical background in Ince \cite{Ince:book}. In general, these equations are irreducible, in the sense that they cannot be solved in terms of elementary functions or known classical special functions (Airy, Bessel, hypergeometric functions, etc.). Moreover, their solutions, namely the \emph{Painlev\'e transcendents}, satisfy the famous \emph{Painlev\'e property}: the only movable singularities are poles. Here, ``movable" means that the location of singularities depends on initial conditions of the equation.

Although the study of the Painlev\'e equations originates from a purely mathematical point of view, they have found important applications in many areas of mathematical physics, such as random matrix theory, statistical physics, integrable continuous dynamical systems,  nonlinear waves, 2D quantum gravity, etc. In addition, the Painlev\'e transcendents themselves also satisfy a lot of nice properties; see the nice survey article Clarkson \cite{Clarkson2006} and references therein. Nowadays, the Painlev\'e transcendents are recognized as nonlinear special functions in the 21st century; see the NIST handbook \cite{Clarkson:NIST}.

In the study of the Painlev\'e equations, one of the important topics is the asymptotic behavior of their solutions. Let us focus on PII, which is probably the most well studied one among the Painlev\'e family. Denote the solution of PII by $u(x;\alpha)$. It is well-known that the PII transcendents are meromorphic functions whose poles are all simple with residue $\pm 1$; for example see Gromak, Laine and Shimomura \cite[Sec. 2]{Gromak2002}. In general, $u(z;\a)$ possesses infinitely many poles which are distributed in the complex $z$-plane; see the nice numerical plots by Fornberg and Weideman \cite{Fornberg2014}. There also exist ``tronqu\'ee" and ``tri-tronqu\'ee" solutions whose poles are confined in some sectors when $z$ is large enough; see Joshi and Mazzocco \cite{Joshi:Maz}. If the solutions are pole-free on the real line, they usually have significant applications in mathematical physics. For these pole-free solutions, it is important to study their asymptotic behaviors as $x \to \pm \infty$ and find out how these behaviors are related, i.e. the \emph{connection formulas}. For the homogeneous case $\a = 0$, these solutions have been studied intensively in the literature.

\subsection{Homogeneous PII}

When $\a = 0$, let us consider solutions of \eqref{PII-def} which decay as $x \to +\infty$, that is, the solutions have the following boundary condition
\begin{equation} \label{PII0-decay-bc}
  u(x;0) \to 0 \qquad \textrm{as }  x \to + \infty.
\end{equation}
With the above boundary condition, one can see that the term $2u^3$ is small and negligible comparing with the other term $xu$. Then, the equation \eqref{PII-def} is close to the Airy equation $u''= x u$ when $x \to +\infty$. Indeed, Hastings and McLeod \cite{Hastings1980} proved the following results.
\begin{theorem}
  When $ \a = 0$, any solution of \eqref{PII-def} satisfying \eqref{PII0-decay-bc} is asymptotic to $k \Ai(x)$, for some $k$, with
$\Ai(x)$ the Airy function. Conversely, for any $k$, there is a unique solution which is
asymptotic to $k \Ai(x)$ as $x \to + \infty$, for some $k$.
\end{theorem}
Let us first consider the real solutions, i.e., $u(x) \in \mathbb{R}$ for real $x$. It is well-known that, depending on the values of $k \in \mathbb{R}$, there exist three classes of solutions for $k \in (-1,1)$, $k = \pm 1$ and $|k| > 1$, respectively.

\medskip

\noindent \textbf{Ablowitz-Segur(AS) solutions: $k \in (-1,1)$.} \\ The AS solution $u_{\textrm{AS}}(x;0)$ is a one-parameter family of solutions of the homogeneous PII in \eqref{PII-def}. They are continuous on the real axis and possess the following asymptotic behaviors.
\begin{eqnarray}
  u_{\textrm{AS}}(x;0) & = & k \Ai (x) (1+o(1)), \quad k \in (-1,1), \qquad \textrm{as } x \to +\infty, \label{PII0-AS1} \\
  u_{\textrm{AS}}(x;0) & = & \frac{d}{(-x)^{1/4}} \cos \biggl(\frac{2}{3}(-x)^{3/2}-\frac{3}{4}d^2\ln(-x) + \phi \biggr) +O\left(\frac{\ln(-x)}{(-x)^{5/4}}\right),  \label{PII0-AS2} \\
  && \hspace{6.6cm} \textrm{as } x \to -\infty, \nonumber
\end{eqnarray}
where the constants $d$ and $\phi$ satisfy the following connection formulas
\begin{eqnarray}
  d(k) &=& \frac{1}{\sqrt{\pi}} \sqrt{-\ln(1-k^2)}, \label{PII0-AS-conn1} \\
  \phi(k) &=& -\frac{3}{2} d^2 \ln 2 +\arg \Gamma{\biggr(\frac{1}{2}id^2}\biggr)+ \frac{\pi}{2} \textrm{sgn}\, k - \frac{\pi}{4} . \label{PII0-AS-conn2}
\end{eqnarray}

\noindent \textbf{Hastings-McLeod(HM) solutions: $k=1$.} \\ The HM solution $u_{\textrm{HM}}(x;0)$ is the unique solution of the homogeneous PII in \eqref{PII-def}, which is continuous on the real axis and has the following asymptotic behaviors
\begin{eqnarray}
  u_{\textrm{HM}}(x;0) & = & \Ai (x) (1+o(1)),  \hspace{1.8cm} \textrm{as } x \to +\infty, \label{PII0-HM1} \\
  u_{\textrm{HM}}(x;0) & = & \sqrt{\frac{-x}{2}} + O\left(\frac{1}{(-x)^{5/2}}\right), \quad \textrm{as } x \to -\infty. \label{PII0-HM2}
\end{eqnarray}

\begin{rmk}
  When $k \to \pm 1$ in \eqref{PII0-AS1}, the oscillatory behavior \eqref{PII0-AS2} as $x \to -\infty$  turns into the square root behavior $u(x;0) \sim \textrm{sgn}(k) \sqrt{-x/2}$. In the literature, the HM solution given in \eqref{PII0-HM1} and \eqref{PII0-HM2} corresponds to the case when $k=1$. One can easily get the solution for $k=-1$ through the following symmetry relation
  \begin{equation}\label{sym-transfmn}
    u(x;\a)=-u(x;-\alpha) \qquad \textrm{for all } \a \in \mathbb{C}.
  \end{equation}
\end{rmk}


\noindent \textbf{Singular solutions: $|k|>1$, $k \in \mathbb{R}$.} \\
When $|k|>1$, the solution $u(x;0)$ is no longer pole-free on the real line and infinitely many poles appear on the negative real axis; see the numerical plot in Fornberg and Weideman \cite[Fig. 12]{Fornberg2014}. These solutions have the following asymptotic behaviors
\begin{eqnarray}
  u_{\textrm{SIN}}(x;0) & = & k \Ai (x) (1+o(1)), \quad |k|> 1, \ k \in \mathbb{R}, \qquad \textrm{as } x \to +\infty, \\
  u_{\textrm{SIN}}(x;0) &=& \frac{\sqrt{-x}}{\sin\biggl(\frac{2}{3}(-x)^{3/2} + \beta \ln(8(-x)^{3/2}) + \phi \biggr) + O((-x)^{-3/2})} \nonumber \\
  &&  + O((-x)^{-1}),  \qquad \textrm{as } x \to -\infty, \label{PII0-sin}
\end{eqnarray}
 where $x$ is bounded away from the singularities appearing in the denominator. Here, the constants $\beta$ and $\phi$ satisfy the following connection formulas
  \begin{eqnarray}
      \beta(k) &=& \frac{1}{2\pi} \ln(k^2-1),  \\
      \phi(k) &=& -\arg \Gamma\biggr(\frac{1}{2} + i \beta \biggr)  + \frac{\pi}{2} \textrm{sgn}\, k - \frac{\pi}{2}. \label{PII0-sin-conn}
  \end{eqnarray}

Besides the above cases where $k$ is real, the parameter $k$ could be complex, too. In the literature, people are also interested in the purely imaginary solutions, i.e., $\re u = 0$ for $x \in \mathbb{R}$. Since the residues of all poles of PII transcendents are 1 or $-1$, the purely imaginary solution is pole-free for all $k \in i\mathbb{R}$. Therefore, unlike the above three different cases for real $k$, there is only one case when $k\in i\mathbb{R} $ and the behavior of the solutions is similar to that of the real AS solutions. Here, we adopt the notation ``$i$AS" for the purely imaginary case and hope this will not bring any confusion.

\medskip
\noindent \textbf{Purely imaginary Ablowitz-Segur($i$AS) solutions: $k \in i \mathbb{R}$.} \\
The $i$AS solution $u_{i\textrm{AS}}(x;0)$ is a one-parameter family of solutions of the homogeneous PII in \eqref{PII-def} with the following asymptotic behaviors
\begin{eqnarray}
  u_{i\textrm{AS}}(x;0) & = & k \Ai (x) (1+o(1)), \quad k \in i\mathbb{R}, \qquad \textrm{as } x \to +\infty, \label{PII0-iAS1} \\
  u_{i\textrm{AS}}(x;0) & = & \frac{d}{(-x)^{1/4}} \sin \biggl(\frac{2}{3}(-x)^{3/2}-\frac{3}{4}d^2\ln(-x) + \phi \biggr) +O\left(\frac{\ln(-x)}{(-x)^{5/4}}\right),  \label{PII0-iAS2} \\
  && \hspace{6.6cm} \textrm{as } x \to -\infty, \nonumber
\end{eqnarray}
where the constants $d$ and $\phi$ satisfy the following connection formulas
\begin{eqnarray}
  d(k) &=& \frac{i}{\sqrt{\pi}} \sqrt{\ln(1+|k|^2)}, \label{PII0-iAS-conn1} \\
  \phi(k) &=& -\frac{3}{2} d^2 \ln 2 +\arg \Gamma{\biggr(\frac{1}{2}id^2}\biggr)- \frac{\pi}{4}. \label{PII0-iAS-conn2}
\end{eqnarray}

The beautiful formal expansions for the AS solution in \eqref{PII0-AS1}-\eqref{PII0-AS-conn2} and the HM solutions in  \eqref{PII0-HM1}-\eqref{PII0-HM2} were derived by Ablowitz and Segur in \cite{Ablowitz1977-asymptotic,Segur1981} and Hastings and McLeod in  \cite{Hastings1980}, respectively. Later, these results were justified rigorously in several different ways, for example, through Gelfand-Levitan type integral equations \cite{Clarkson1988,Hastings1980}, the isomonodromy method \cite{Its:Kap1988,Suleimanov1987}, the Deift-Zhou nonlinear steepest descent method \cite{Deift1995} and a uniform approximation method \cite{Bas:Cla:Law:McL1998}. The singular asymptotics in \eqref{PII0-sin}-\eqref{PII0-sin-conn} are relatively new comparing with the AS and HM asymptotics. They were first derived by Kapaev through the isomonodromy method in \cite{Kapaev1992}, and then proved again by Bothner and Its in \cite{Bot:Its2012} with the Deift-Zhou nonlinear steepest descent method. For the purely imaginary case, the asymptotic expansions in \eqref{PII0-iAS1}-\eqref{PII0-iAS2} and the corresponding connection formulas \eqref{PII0-iAS-conn1}-\eqref{PII0-iAS-conn2} were also obtained by Its and Kapaev \cite{Its:Kap1988} with the isomonodremy method and justified again by Deift and Zhou \cite{Deift1995} with the nonlinear steepest descent method. It is interesting to note that both the isomonodromy method and Deift-Zhou nonlinear steepest descent method are based on the Riemann-Hilbert(RH) problem for the Painlev\'e equations; see the RH problem for PII in Section \ref{sec:rhp} as an example. The RH problem originates from the pioneer work of Flaschka and Newell \cite{Flaschka1980} and Jimbo, Miwa and Ueno \cite{Jimbo1981}, where they considered a system of linear ordinary differential equations with regular and irregular singular points.  The main difference between these two methods is that certain prior assumptions on the asymptotics of the solutions are needed in the isomonodromy method. Therefore, one may view the Deift-Zhou nonlinear steepest descent method as a direct method in the sense that it does not need any prior assumptions.

\begin{rmk}
  Recently, the transition asymptotics for the homogeneous PII was studied by Bothner \cite{Bothner2015}. The transition asymptotics describes how the asymptotic behavior of $u(x;0)$ as $x \to -\infty$ changes from an oscillatory one (Ablowitz and Segur) to a power-like one (Hastings and McLeod), and then to a singular one (Kapaev) when the parameter $k$ in \eqref{PII0-AS1} varies. Using the Deift-Zhou nonlinear steepest descent method, the author shows that the transition asymptotics are of Boutroux type, i.e., they are expressed in terms of Jacobi elliptic functions.
\end{rmk}

\begin{rmk}
  For the homogeneous PII, Bogatskiy, Claey and Its \cite{Bogatskiy:Claey:Its2016} extended the real and purely imaginary AS solutions to \emph{complex AS solutions}  by letting the parameter $k$ in \eqref{PII0-AS1} and \eqref{PII0-iAS1} be a complex number in $\mathbb{C}\cut \left((-\infty, -1] \cup [1, +\infty)\right)$. Similar results for the asymptotics, connection formulas as well as the pole-free property on the real axis are also obtained; see \cite[Thm. 1]{Bogatskiy:Claey:Its2016}.

\end{rmk}

\subsection{Inhomogeneous PII and our main results}

Although there is an extensive literature regarding the asymptotics for the homogeneous PII, the asymptotic results for the inhomogeneous PII are relatively fewer. When $\a \neq 0$, if one also seeks a solution like \eqref{PII0-decay-bc} which decay as $x \to +\infty$, then we need to balance terms $xu$ and $\a$ in \eqref{PII-def}. That is, when $\a \neq 0$, the solutions $u(x;\a)$ are no longer exponentially small, but tend to 0 like $\a / x$. Then, to obtain real and purely imaginary solutions, the parameter $\a$ is required to be real and purely imaginary, respectively.

For the real case ($\a \in \mathbb{R}$), the inhomogeneous PII also possesses solutions which are similar to the AS and HM solutions given in \eqref{PII0-AS1}-\eqref{PII0-HM2}.
In this paper, we are going to prove the following results for $u_{\textrm{AS}}(x;\a)$ rigorously.

\begin{theorem}\label{main-thm}
    Given $\a\in (-\frac{1}{2}, \frac{1}{2})$, there exists a one-parameter family of real solutions $u_{\textrm{AS}}(x;\a)$ for $ k \in (-\cos(\pi \a),  \cos(\pi \a))$ with the following properties:
    \begin{itemize}

    \item[(a)] $u_{\textrm{AS}}(x;\a)$ has the following asymptotic behaviors:
    \begin{equation}\label{asy-pos}
    u_{\textrm{AS}}(x;\a)= B(\alpha; x) + k\Ai(x)(1+O(x^{-3/4})), \qquad \textrm{as} \,\, x \to +\infty,
    \end{equation}
    and
    \begin{equation}\label{asy-neg}
    u_{\textrm{AS}}(x;\a)=\frac{d}{(-x)^{1/4}}\cos\{\frac{2}{3}(-x)^{3/2}-\frac{3}{4}d^2\ln(-x) + \phi\} +O(|x|^{-1}), \quad \textrm{as}\,\, x\to -\infty,
    \end{equation}
    where $\Ai(x)$ is the Airy function and $B(\alpha; x)$ is of the form
    \begin{align}\label{asy-B(a;x)}
    B(\alpha; x)\sim \frac{\alpha}{x}\sum_{n=0}^{\infty} \frac{a_n}{x^{3n}}.
    \end{align}
    The coefficients $a_n$ in the above formula are determined uniquely through the following recurrence relation
    \begin{align}\label{a_n-rec-rel}
     a_0 = 1, \qquad a_{n+1}=(3n+1)(3n+2)a_n - 2\a^2 \sum_{k,l,m =0}^{n} a_k a_l a_m.
    \end{align}

    \item[(b)] The connection formulas are given by
    \begin{equation}\label{d-k}
    d(k) = \frac{1}{\sqrt{\pi}} \sqrt{-\ln(\cos^2(\pi \alpha)-k^2)},
    \end{equation}
    and
    \begin{equation}\label{phi-k}
    \phi(k) = -\frac{3}{2} d^2 \ln 2 +\arg \Gamma{\biggr(\frac{1}{2}id^2}\biggr)- \frac{\pi}{4} -\arg (-\sin \pi \a -ki).
    \end{equation}

    \item[(c)] $u_{\textrm{AS}}(x;\a)$ is pole-free on the real line.
    \end{itemize}
\end{theorem}

  The formal asymptotic behaviors \eqref{asy-pos} and \eqref{asy-neg}, as well as the connection formulas \eqref{d-k} and \eqref{phi-k} first appeared in McCoy and Tang \cite{mccoy:Tang1986}. Later, Kapaev \cite{Kapaev1992} rigorously justified these results by using  the isomonodromy method. Unfortunately,  the detailed proofs were not provided in \cite{Kapaev1992}. The rigorous proof of \eqref{asy-pos} as $x \to + \infty$ can be found in Fokas et al. \cite[Chap. 11]{Fokas2006} and Its and Kapaev \cite{Its2003}, which relies on the Deift-Zhou nonlinear steepest descent method for Riemann-Hilbert problems.
  However, to the best of our knowledge, the rigorous proofs of the asymptotic behavior as $x \to -\infty$, the connection formulas \eqref{d-k} and \eqref{phi-k}, and the pole-free property of $u_{\textrm{AS}}(x;\a)$ have never appeared in the literature. In this paper, we will derive the asymptotic results rigorously by using the Deift-Zhou nonlinear steepest descent method, and prove the pole-free property by establishing a vanishing lemma for the associated RH problems.

\begin{rmk}
  Similar to the homogeneous case, when $k \to \cos(\pi \a)$, the oscillatory behavior \eqref{asy-neg} as $x \to -\infty$  is replaced by the square root behavior $u(x;\a) \sim \textrm{sgn}(k) \sqrt{-x/2}$. The solution with these boundary conditions is the so-called HM solution $u_{\textrm{HM}}(x;\a)$ for the inhomogeneous PII in the literature. The asymptotic behaviors of $u_{\textrm{HM}}(x;\a)$ have been obtained in Its and Kapaev \cite{Its2003} and Kapaev \cite{Kapaev2004}. Later, Claeys, Kuijlaars and Vanlessen \cite{Claeys2008} proved that,  when $\a > -1/2$, the HM solution is pole-free for all $x \in \mathbb{R}$.
\end{rmk}

\begin{rmk}
  When $k> \cos(\pi \a)$ or $k < -\cos(\pi \a)$, infinitely many poles will also appear on the negative real axis. This is similar to the singular solutions $u_{\textrm{SIN}}(x;0)$ mentioned for the homogeneous PII. We wish to apply the Deift-Zhou nonlinear steepest descent method to study the asymptotic properties of these solutions in the near future.
\end{rmk}

\begin{rmk}
  In Theorem \ref{main-thm}, we show that $u_{\textrm{AS}}(x;\a)$ is pole-free on the real line for $\a \in (-1/2, 1/2)$. However, Claeys, Kuijlaars and Vanlessen \cite{Claeys2008} proved that $u_{\textrm{HM}}(x;\a)$ is pole-free on the real line for $\a > -1/2$. One may wonder whether it is possible to extend the validity range of the parameter $\a$ for $u_{\textrm{AS}}(x;\a)$. We believe $(-1/2, 1/2)$ is the largest possible range for $\a$ based on our proof as well as the numerical evidence in Fornberg and Weideman \cite{Fornberg2014}. In \cite[Fig. 5]{Fornberg2014}, the authors provided a very interesting pole counting diagrams in the $(u(0),u'(0))$-plane. In the diagrams, curves labeled $n^+$ denote initial conditions that generate solutions with $n$ poles on $\mathbb{R}^+$, while curves and regions labeled $n^-$ represent $n$ poles on $\mathbb{R}^-$. One can see that, only when $\a \in (-1/2, 1/2)$, the curve $0^+$ passes through the region $0^-$, which indicates the initial conditions corresponding to $u_{\textrm{AS}}(x;\a)$.
\end{rmk}

Similar to the purely imaginary solutions $u_{i\textrm{AS}}(x;0)$ in the homogeneous case, there also exist purely imaginary solutions $u_{i\textrm{AS}}(x;\a)$ when $\a \in i \mathbb{R}$.
Although there is a substantial literature on the Painlev\'e equations, we have not been able to find the following asymptotic results.

\begin{theorem}\label{main-thm-im}
Given $\a \in i\mathbb{R}$, there exists a one-parameter family of purely imaginary solution $u_{\textrm{iAS}}(x;\alpha)$ for $k\in i\mathbb{R}$ with the following properties:
\begin{itemize}

\item[(a)] $u_{i\textrm{AS}}(x;\alpha)$ has the following asymptotic behaviors:
\begin{equation}\label{asy-pos-im}
u_{i\textrm{AS}}(x;\alpha)= B(\alpha; x) +  k \Ai(x)(1+O(x^{-3/4})), \qquad \textrm{as} \,\, x \to +\infty
\end{equation}
and
\begin{equation}\label{asy-neg-im}
u_{i\textrm{AS}}(x)=\frac{d}{(-x)^{1/4}} \sin \{\frac{2}{3}(-x)^{3/2}-\frac{3}{4}d^2\ln(-x) + \phi\} +O(|x|^{-1}), \quad \textrm{as}\,\, x\to -\infty,
\end{equation}
where $B(\alpha; x)$ is defined by \eqref{asy-B(a;x)} and \eqref{a_n-rec-rel}.

\item[(b)] The connection formulas are given by
\begin{equation}\label{d-k-im}
d(k) = \frac{i}{\sqrt{\pi}} \sqrt{\ln(\cosh^2(\pi i \a)+|k|^2)}
\end{equation}
and
\begin{equation}\label{phi-k-im}
\phi(k) = -\frac{3}{2} d^2 \ln 2 +\arg \Gamma{\biggr(\frac{1}{2}id^2}\biggr)- \frac{\pi}{4} -\arg (-ik+i \sinh{(\pi i \a)}).
\end{equation}

\item[(c)] $u_{i\textrm{AS}}(x;\alpha)$ is pole-free on the real line.
\end{itemize}
\end{theorem}

\begin{rmk}
  In the above theorem, if one takes $\a=0$, it is straightforward to see that the asymptotic expansions and connection formulas are reduced to those for $u_{i\textrm{AS}}(x;0)$ listed in \eqref{PII0-iAS1}-\eqref{PII0-iAS-conn2}. Although the formulas \eqref{asy-pos-im}-\eqref{phi-k-im} and \eqref{PII0-iAS1}-\eqref{PII0-iAS-conn2} are very similar, we haven't found any other formal or rigorous results in the literature, except for the recent numerical simulations by  Fornberg and Weideman in \cite{Fornberg2015}.
\end{rmk}

The rest of the paper is arranged as follows. In Section \ref{sec:rhp}, we list the RH problem for PII and derive the conditions for the real and purely imaginary solutions. We also prove the vanishing lemma for the RH problem which is associated with the real solution $u_{\textrm{AS}}(x;\a)$ when $\a\in (-\frac{1}{2}, \frac{1}{2})$. In Section \ref{sec:thm2-bc:thm3}, we conduct the Deift-Zhou nonlinear steepest descent method for the RH problem of PII as $x \to -\infty$. Finally, in Section \ref{sec-main-proof}, we prove our main results.

\section{The RH problem for PII} \label{sec:rhp}

Let $L = \cup_{k=1}^6 l_k$  be the contour in the complex $\l$-plane, consisting of six rays oriented to infinity and
\begin{eqnarray*}
l_k: \quad \arg \l = \frac{\pi}{6} + \frac{(k-1)\pi}{3}, \qquad k = 1 , 2, \cdots, 6.
\end{eqnarray*}
Let $\Psi_\a(\l) = \Psi_\a(\l;x)$ be a $2 \times 2$ matrix-valued function, where the notation $x$ is usually suppressed for brevity. Then, the solution $u (x; \a)$ of PII is associated with the following RH problem for $\Psi_\a(\l)$; see Its and Kapaev \cite{Its2003}.
\begin{figure}[h]
\centering
  \includegraphics[width=8cm]{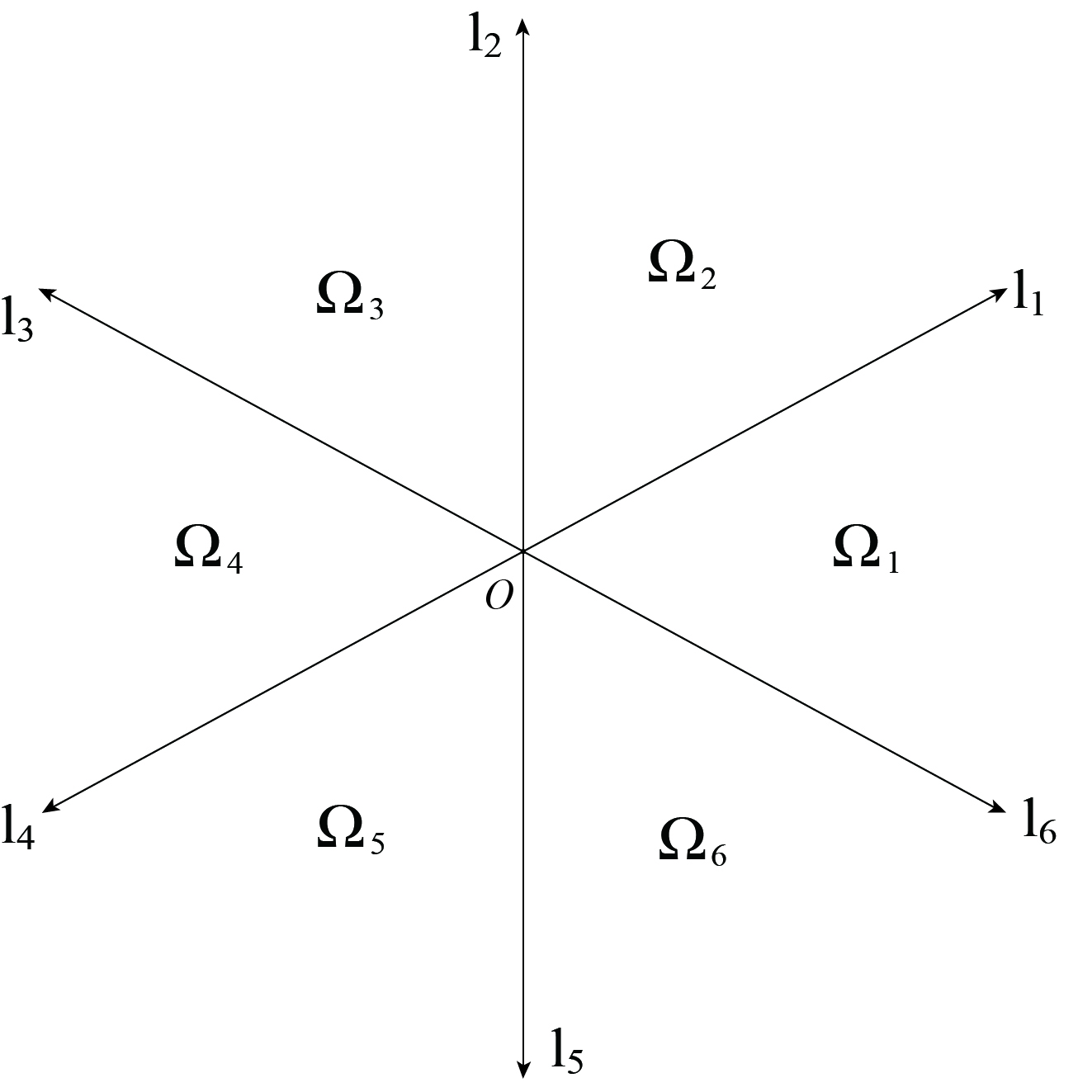}\\
  \caption{The contour $L$ for the RH problem associated with PII}\label{contour-Arno}
\end{figure}

\begin{itemize}
\item[(a)] $\Psi_\a(\l)$ is analytic for $\l \in \mathbb{C} \cut L$; see Figure \ref{contour-Arno};

\item[(b)]  Let $\Psi_{\a, \pm}(\l)$ denote the limiting values of $\Psi_\a(\l)$ as $\l$ tends to the contour $L$
from the left and right sides, respectively. Then, $\Psi_\a(\l)$ satisfies the following jump conditions:
\begin{equation} \label{Psi+Psi-}
\Psi_{\a, +}(\l)= \Psi_{\a, -} (\l) S_k, \qquad \textrm{for } \l \in l_k
\end{equation}
with
\begin{equation} \label{Psi-jumps}
 S_k = \left(
  \begin{array}{cc}
  1 & 0\\
    s_k & 1
  \end{array}
    \right), \quad k=1, 3, 5, \quad \textrm{and}   \quad S_k = \left(
  \begin{array}{cc}
  1 & s_k\\
    0 & 1
  \end{array}
    \right),   \quad k=2, 4, 6.
\end{equation}
The constants $s_k$'s are called the \emph{Stokes multipliers}, which satisfy the following constraints
    \begin{align}\label{iso-data-con}
     s_{k+3} = -s_k, \qquad s_1 - s_2 +s_3 + s_1 s_2 s_3  = - 2 \sin \pi\a;
    \end{align}

 \item[(c)] As $\l \to \infty$, $\Psi_\a(\l)$ has the following asymptotic behavior:
 \begin{equation}\label{Psi-infty-asy}
  \Psi_\a(\l) = \biggl[ I+ \frac{\Psi_{-1}(x)}{\l}+O(\l^{-2}) \biggr]e^{-\theta (\l)\sigma_3} , 
 \end{equation}
 where
 \begin{equation} \label{theta-def}
   \theta(\l)=\theta(\l;x) = i(\frac{4}{3}\l^3 + x\l)
 \end{equation}
 and $\sigma_3$ is the Pauli matrix $\left(\begin{array}{cc}
  1 & 0\\
    0 & -1
  \end{array}\right)$;

 \item[(d)] At $\l =0$, $\Psi_\a(\l)$ has the following behaviors:
 \begin{equation}\label{Psi-origin-asy}
  \Psi_\a(\l) = O\left(\begin{matrix}
  |\l|^{-\a} &  |\l|^{-\a}\\
   |\l|^{-\a} & |\l|^{-\a}
  \end{matrix}\right), \qquad \textrm{if } 0<  \re \alpha < \frac{1}{2}
\end{equation}
and
\begin{equation}\label{Psi-origin-asy-alpha-neg}
  \Psi_\a(\l) =
  O\left(\begin{matrix}
  |\l|^{\a} &  |\l|^{\a}\\
   |\l|^{\a} & |\l|^{\a}
  \end{matrix}\right), \qquad \textrm{if } -\frac{1}{2} < \re \alpha \leq 0.
\end{equation}
\end{itemize}

\begin{rmk}
  Since we only consider the real and purely imaginary Ablowitz-Segur solutions in this paper, it is enough for us to require $\re \a \in (-\frac{1}{2}, \frac{1}{2})$ in \eqref{Psi-origin-asy} and \eqref{Psi-origin-asy-alpha-neg}. Of course, one may consider $\a$ as a general complex constant. In this situation, the behavior near $\l = 0$ needs to be slightly modified. Typically, when $\a - \frac{1}{2} \in \mathbb{Z}$, a logarithmic singularity will take place; see Its and Kapaev \cite{Its2003}.
\end{rmk}

From the above RH problem for $\Psi_\a$, one may consider properties of the functions $ \D  \frac{\partial \Psi_\a}{\partial \l} \Psi_\a^{-1}$ and $  \D \frac{\partial \Psi_\a}{\partial x}\Psi_\a^{-1}$.
Because the jump matrices $S_k$ in \eqref{Psi-jumps} are independent of the variables $\l$ and $x$, these two functions are meromorphic functions in the complex $\l$-plane with only possible isolated singularities at $\l = 0$ and $\l = \infty$. Indeed, these two functions can be computed explicitly, which gives us the following \emph{Lax pair} for PII:
\begin{equation}\label{Lax-pair-A}
 \frac{\partial \Psi_\a}{\partial \lambda} = \mathcal{A}(\l) \Psi_\a, \quad \mathcal{A}(\l)=-i(4\lambda ^2 + x + 2u^2)\sigma_3 -(4\lambda u +\frac{\alpha}{\lambda}) \sigma_2 - 2u_x \sigma_1,
\end{equation}
and
\begin{equation}\label{Lax-pair-U}
  \frac{\partial \Psi_\a}{\partial x} = \mathcal{U}(\l) \Psi_\a, \quad \mathcal{U}(\l)=-i\lambda \sigma_3 - u\sigma_2,
\end{equation}
where $\sigma_1 = \left(\begin{array}{cc}
  0 & 1\\
    1 & 0
  \end{array}\right)$ and
  $\sigma_2 = \left(\begin{array}{cc}
  0 & -i\\
    i & 0
  \end{array}\right)$. Moreover, one can prove that
\begin{equation}
  u (x; \a) =2 \biggl( \Psi_{-1}(x) \biggr)_{12}
\end{equation}
is a solution of the PII equation \eqref{PII-def}; see Fokas et al. \cite[Theorem 5.1]{Fokas2006} for more details.

\begin{rmk}
The RH problem for $\Psi_\a$ in this section is a little different from those for PII in \cite[p.616]{Claeys2008} or \cite[Theorem 5.1]{Fokas2006}. If one makes the following transformation
\begin{equation}
\Psi^*(\l) =  e^{\frac{\pi i}{4}\sigma_3}\Psi_\a(\l)e^{-\frac{\pi i}{4}\sigma_3},
\end{equation}
then $\Psi^*(\l)$ satisfies the same RH problem as those in \cite{Claeys2008,Fokas2006}. Note that the above transformation changes the constants in \eqref{Psi-jumps} and \eqref{iso-data-con} such that $s_k^* =  (-1)^{k} i s_k$ and $\a^* = -\a$.
\end{rmk}

\subsection{The real and purely imaginary solutions of PII}

It is well-known that, the map
\begin{equation}
  \{(s_1, s_2, s_3 ) \in \mathbb{C}^3: s_1 - s_2 +s_3 + s_1 s_2 s_3  = - 2 \sin \pi\a \} \mapsto \{ \textrm{solutions of PII}\}
\end{equation}
is a bijection; for example, see \cite[Chap. 4]{Fokas2006}. This means that one can obtain some special solutions of PII \eqref{PII-def} by choosing certain special Stokes multipliers. In this paper, we will focus on the real and purely imaginary solutions $u(x;\a)$ when $\a \in (-\frac{1}{2} ,\frac{1}{2})$ and $\a \in i\mathbb{R}$, respectively. Therefore, we wish to know what kind of Stokes multipliers will lead to our desired real and purely imaginary solutions. For this purpose, let us first derive the following result.

\begin{prop}\label{Phi=Psi-prop}
    Let $\a\in( -\frac{1}{2}, \frac{1}{2})$ and $x \in \mathbb{R}$.  Then, $\Psi_\a(\l):=\sigma_1\overline{\Psi_\a(\bar{\l})}\sigma_1$ if and only if
    \begin{equation} \label{prop1-condition}
      s_3=\bar{s}_1 \qquad \textrm{and} \qquad  s_2=\bar{s}_2,
    \end{equation}
    where $s_k$'s are the Stokes multipliers given in \eqref{Psi-jumps} and \eqref{iso-data-con}.
\end{prop}

\begin{proof}

Let us denote
\begin{equation}
  \Phi_\a(\l) = \Phi_\a(\l; x) :=\sigma_1\overline{\Psi_\a(\bar{\l})}\sigma_1.
\end{equation}
First, we will show that the condition \eqref{prop1-condition} is sufficient for $\Phi_\a(\l) = \Psi_\a(\l)$. According to standard arguments based on Liouville's theorem, it is easily seen that the solution to the RH problem for $\Psi_\a$ must be unique. Then, it is enough to verify that $\Phi_\a$ satisfies the same RH problem as $\Psi_\a$.

Since the contour $L$ is symmetric about the real axis (see Figure \ref{contour-Arno}), it is obvious that $\Phi_\a(\l)$ is analytic for $\l \in \mathbb{C} \setminus L.$ Next, notice that
\begin{equation*}
  \l \in l_k \Longleftrightarrow \bar\l \in l_{7-k}, \qquad k = 1, 2, \cdots, 6
\end{equation*}
and
\begin{equation*}
  \Phi_{\a,\pm}(\l) = \sigma_1\overline{\Psi_{\a, \mp}(\bar{\l})} \sigma_1 \qquad \textrm{for } \l \in l_k.
\end{equation*}
Then, we have
\begin{equation} \label{Phi-jumps}
\Phi_{\a,-}^{-1}(\l) \Phi_{\a,+}(\l)= \sigma_1 \overline{\Psi_{\a, +}(\bar{\l}) }^{-1} \; \overline{\Psi_{\a, -}(\bar{\l}) }\sigma_1 = \sigma_1 \bar{S}_{7-k}^{-1}\sigma_1 \qquad \textrm{for } \l \in l_k.
\end{equation}
Recalling $s_{k+3} = - s_k$ in \eqref{iso-data-con} and the conditions $s_3=\bar{s}_1$ and $s_2= \bar{s}_2$ in \eqref{prop1-condition}, we obtain
$\sigma_1 \bar{S}_{7-k}^{-1}\sigma_1 = S_k$, which indicates that $\Phi_\a(\l)$ satisfies the same jump conditions as $\Psi_\a(\l)$ for $\l \in l_k$. Finally, as $\overline{\theta(\bar \l;x)} = - \theta(\l;x)$ for real $x$, it is easy to check that $\Phi_\a(\l)$ also has the same asymptotic behaviors as $\Psi_\a(\l)$ in \eqref{Psi-infty-asy}-\eqref{Psi-origin-asy-alpha-neg}. So, $\Phi_\a$ indeed satisfies the same RH problem as $\Psi_\a$.

Next, we will show that \eqref{prop1-condition} is also a necessary condition for $\Phi_\a(\l) = \Psi_\a(\l)$. From jump conditions \eqref{Psi+Psi-} and \eqref{Phi-jumps} for $\Psi_\a$ and $\Phi_\a$, we have $\sigma_1 \bar{S}_{7-k}^{-1}\sigma_1 = S_k$. Then, the explicit formulas of $S_k$'s in \eqref{Psi-jumps} and the relation in \eqref{iso-data-con} gives us \eqref{prop1-condition}.

This completes the proof of our proposition.
\end{proof}

With the above proposition, we are ready to prove the following necessary and sufficient condition for a real solution $u(x; \a)$ when $\a \in (-\frac{1}{2}, \frac{1}{2})$.
\begin{theorem}\label{real-reduction-thm}
    For $\a \in (-\frac{1}{2}, \frac{1}{2})$, the solution $u(x; \a)$ of PII \eqref{PII-def} is real for $x \in \mathbb{R}$ if and only if $s_3 = \bar{s}_1$ and $s_2=\bar{s}_2$.
\end{theorem}

\begin{proof}

First of all, we note that, when $\a \in (-\frac{1}{2}, \frac{1}{2})$,
\begin{equation} \label{u-real-for1}
  u(x; \a) \textrm{ is real for }  x \in \mathbb{R} \Longleftrightarrow \sigma_1 \overline{\mathcal{A}(\bar{\l})} \sigma_1 = \mathcal{A}(\l);
\end{equation}
see the definition of $\mathcal{A}(\l)$ in \eqref{Lax-pair-A}. Moreover, the following formula is also useful in our subsequent proof:
\begin{equation} \label{u-real-for2}
    \frac{\partial{[\sigma_1 \overline{\Psi_\a(\bar \l})} \sigma_1]}{\partial \l}= \sigma_1\overline{\mathcal{A}(\bar \l)}\sigma_1 [\sigma_1 \overline{\Psi_\a(\bar{\l})}\sigma_1],
\end{equation}
which comes from \eqref{Lax-pair-A} and the fact $\D\overline{\left. \frac{\partial[{\Psi}_\a(\l)]}{\partial{\l}}\right|_{\l=\bar{\l}}}=\frac{\partial [\overline{\Psi_\a(\bar{\l})} ]}{\partial \l}$.

Let us first prove the sufficient part. If $s_3=\bar{s}_1$ and $s_2=\bar{s}_2$, Proposition \ref{Phi=Psi-prop} gives us $\sigma_1\overline{\Psi_\a(\bar{\l})}\sigma_1 = \Psi_\a(\l)$.
Replacing $\Psi_\a(\l)$ in \eqref{Lax-pair-A} by $\sigma_1\overline{\Psi_\a(\bar{\l})}\sigma_1$, we have
\begin{equation}\label{Phi-A}
    \frac{\partial {[\sigma_1\overline{\Psi_\a(\bar \l)}\sigma_1]}}{\partial \l} = \mathcal{A}(\l)[\sigma_1\overline{\Psi_\a(\bar \l)}\sigma_1].
\end{equation}
Then, the above formula and \eqref{u-real-for2} yield $ \sigma_1 \overline{\mathcal{A}(\bar{\l})}\sigma_1=\mathcal{A}(\l).$ Thus, $u(x; \a)$ is a real solution when $x\in \mathbb{R}$; see \eqref{u-real-for1}.

Next, we prove the necessary part. If the solution $u(x; \a)$ is a real solution for $x\in \mathbb{R}$, we have $ \sigma_1 \overline{\mathcal{A}(\bar{\l})}\sigma_1=\mathcal{A}(\l)$. It then follows from \eqref{Lax-pair-A} and \eqref{u-real-for2} that $\sigma_1\overline{\Psi_\a(\bar \l)}\sigma_1$ satisfies the same differential equation as $\Psi_\a(\l)$. As $\overline{\theta(\bar \l;x)} = - \theta(\l;x)$ for real $x$, we have
\begin{equation}
\sigma_1\overline{\Psi_\a(\bar{\l})}\sigma_1 = (I + O(\l^{-1}))e^{-\theta(\l)\sigma_3}, \qquad \textrm{as } \l \to \infty,
\end{equation}
which agrees with \eqref{Psi-infty-asy}. So, $\sigma_1\overline{\Psi_\a(\bar{\l})}\sigma_1$ also satisfies the same boundary condition as $\Psi_\a(\l)$. Therefore, we have $\sigma_1\overline{\Psi_\a(\bar{\l})}\sigma_1 = \Psi_\a(\l)$. Using Proposition \ref{Phi=Psi-prop} again, we obtain $s_3=\bar{s}_1$ and $s_2=\bar{s}_2$.

This finishes the proof of Theorem \ref{real-reduction-thm}.
\end{proof}

Similarly, we also derive the necessary and sufficient condition for a purely imaginary solution $u(x; \a)$ when $\a \in i \mathbb{R}$.
\begin{theorem}\label{im-reduction-thm}
    For $\a\in i \mathbb{R}$, the solution $u(x; \a)$ of PII \eqref{PII-def} is purely imaginary for $x \in \mathbb{R}$ if and only if $s_3 = -\bar{s}_1$ and $s_2=-\bar{s}_2$.
\end{theorem}

\begin{proof}
  Comparing with Proposition \ref{Phi=Psi-prop} and Theorem \ref{real-reduction-thm}, one may consider the properties of the function $\sigma_2\overline{\Psi_\a(\bar{\l})}\sigma_2$ when $\a \in i\mathbb{R}$. Then, the rest of the proof is very similar to the analysis in the proofs of Proposition \ref{Phi=Psi-prop} and Theorem \ref{real-reduction-thm}.
\end{proof}

\begin{rmk}
  Indeed, the real and purely imaginary Ablowitz-Segur solutions $u_{\textrm{AS}}(x;\a)$ and $u_{i\textrm{AS}}(x;\alpha)$ correspond to some specific Stokes multipliers. In both cases, $s_2$ is equal to 0. Then, the above two theorems and \eqref{iso-data-con} yield
  \begin{align}\label{stokes-real-im}
    s_1 =  -\sin(\pi \a) - ik, \quad s_2 = 0, \quad  s_3 = -\sin(\pi \a) + ik.
  \end{align}
Note that, when the parameters $\a$ and $k$ satisfy the following conditions
\begin{equation}\label{real-alpha-k}
 \a \in (-\frac{1}{2}, \frac{1}{2}), \qquad k \in (-\cos(\pi \a),  \cos(\pi \a))
\end{equation}
and
\begin{equation}\label{im-alpha-k}
\a \in i\mathbb{R}, \qquad k \in i\mathbb{R},
\end{equation}
the Stokes multipliers in \eqref{stokes-real-im} give us the real solution $u_{\textrm{AS}}(x; \a)$ and the purely imaginary solution  $u_{i\textrm{AS}}(x;\alpha)$, respectively. The above conditions for $\a$ and $k$ are the same as those in Theorems \ref{main-thm} and \ref{main-thm-im}.
\end{rmk}

\subsection{Solvability of the RH Problem for $\Psi_\a(\l)$} 

Before we conduct a nonlinear steepest descent analysis for the RH problem of $\Psi_\a$, let us first prove the following vanishing lemma for the Stokes multipliers given in \eqref{stokes-real-im} under the condition \eqref{real-alpha-k}. This lemma is crucial to prove that the real solution $u_{\textrm{AS}}(x;\a)$ is pole-free on the real axis. The proof is similar to those in \cite{Claeys2008, Fokas1992}. Note that we do not need a vanishing lemma to show that the purely imaginary solution $u_{i\textrm{AS}}(x;\a)$ is pole-free on the real axis. This is due to the fact that the residues of all poles of PII transcendents are 1 or $-1$.

\begin{lemma} \label{lemma-vanishing-1}
Let $\a \in (-\frac{1}{2}, \frac{1}{2})$, $ k \in (-\cos(\pi \a),  \cos(\pi \a))$ and the Stokes multipliers be given in \eqref{stokes-real-im}. Suppose that $\Psi^{(1)}_\a(\l)$ satisfies the same jump conditions \eqref{Psi+Psi-} and the same boundary conditions \eqref{Psi-origin-asy} and \eqref{Psi-origin-asy-alpha-neg} as $\Psi_\a(\l)$, except that the behavior \eqref{Psi-infty-asy} at infinity being altered to
\begin{equation}
\Psi^{(1)}_\a(\l) = O\left(\frac{1}{\l}\right)e^{-\theta(\l)\sigma_3}, \qquad \textrm{as} \quad \l \to\infty.
\end{equation}
Then $\Psi^{(1)}_\a(\l)$ is trivial, that is, $\Psi^{(1)}_{\a}(\l) \equiv 0$.
\end{lemma}
\begin{proof}
We note that $S_2=S_5=I$ when $s_2=0$, so there is no jump on $l_2$ and $l_5$. First, we introduce the following transformation $\Psi^{(1)}_\a \to \Psi^{(2)}_\a$, such that the exponential factor at infinity is removed and $\Psi^{(2)}_\a$ only possesses jumps on $\mathbb{R}$:
\begin{equation} \label{vanishing-psi2-psi1}
\Psi^{(2)}_\a(\l) :=   \begin{cases}\D
\Psi^{(1)}_\a(\l)e^{\theta(\l)\sigma_3}, &\quad \l \in \Omega_2 \cup \Omega_3 \cup \Omega_5 \cup \Omega_6,\\
\Psi^{(1)}_\a(\l)S_1 e^{\theta(\l)\sigma_3}, &\quad \l \in \Omega_1\cap \mathbb{C}_{+},\\
\Psi^{(1)}_\a(\l)S^{-1}_{3} e^{\theta(\l)\sigma_3}, &\quad \l \in \Omega_4 \cap \mathbb{C}_{+},\\
\Psi^{(1)}_\a(\l)S_{4} e^{\theta(\l)\sigma_3}, &\quad  \l \in \Omega_4 \cap \mathbb{C}_{-}, \\
\Psi^{(1)}_\a(\l)S^{-1}_{6} e^{\theta(\l)\sigma_3}, &\quad  \l \in \Omega_1 \cap \mathbb{C}_{-}.
\end{cases}
\end{equation}
It is easy to verify that $\Psi^{(2)}_\a(\l)$ solves a RH problem as follows:
\begin{itemize}
\item[(a)] $\Psi^{(2)}_\a(\l)$ is analytic for $\l \in \mathbb{C} \cut \mathbb{R}$;

\item[(b)] $\Psi^{(2)}_\a(\l)$ satisfies the following jump relations on $\mathbb{R}\cut \{0\}$,
\begin{eqnarray}
\Psi^{(2)}_{\a,+}(\l)=\Psi^{(2)}_{\a,-}(\l)e^{-\theta(\l)\sigma_3} \left(\begin{matrix}
1-s_1 s_3 & s_1\\
-s_3 & 1
\end{matrix}\right)e^{\theta(\l)\sigma_3}, \quad \l \in (-\infty, 0),\label{Psi2-J-neg}\\
\Psi^{(2)}_{\a,+}(\l)=\Psi^{(2)}_{\a,-}(\l)e^{-\theta(\l)\sigma_3}\left(\begin{matrix}
1-s_1 s_3 & -s_3\\
s_1 & 1
\end{matrix}\right)e^{\theta(\l)\sigma_3}, \quad \l \in (0, \infty)\label{Psi2-J-pos};
\end{eqnarray}

\item[(c)]  As $\l \to \infty$, $\Psi^{(2)}_\a(\l)= \D O\left(\frac{1}{\l}\right) $;

\item[(d)]  At $\l = 0$, $\Psi^{(2)}_\a(\l)$ has the following behavior:
\begin{equation}
\Psi^{(2)}_\a(\l) = O\left( \begin{matrix}
                    |\l|^{-\a} & |\l|^{-\a}\\
                    |\l|^{-\a} & |\l|^{-\a}
         \end{matrix}\right), \qquad \textrm{if } 0<\alpha < \frac{1}{2}
\end{equation}
and
\begin{equation} \label{Psi2-origin}
\Psi^{(2)}_\a(\l) = O\left( \begin{matrix}
                    |\l|^{\a} & |\l|^{\a}\\
                    |\l|^{\a} & |\l|^{\a}
         \end{matrix}\right), \qquad \textrm{if } -\frac{1}{2} < \alpha \leq 0.
\end{equation}

\end{itemize}

Next, let us consider the function
\begin{equation}
\Psi^{(3)}_\a(\l)=\Psi^{(2)}_\a(\l)\Psi^{(2)}_\a(\l)^{*}  \qquad \textrm{for} \quad \l \in \mathbb{C}\cut \mathbb{R},
\end{equation}
where $\Psi^{(2)}_\a(\l)^{*}$ denotes the Hermitian conjugate of the matrix $\Psi^{(2)}_\a(\l)$. From the conditions (c) and (d) of the above RH problem for $\Psi^{(2)}(\l)$, we have the following behaviors for $\Psi^{(3)}_\a(\l)$:
\begin{equation}
\Psi^{(3)}_\a(\l) = \begin{cases}
  O\left( \begin{matrix}
                    |\l|^{-2\a} & |\l|^{-2\a}\\
                    |\l|^{-2\a} & |\l|^{-2\a}
         \end{matrix}\right), & \textrm{if } 0<\alpha < \frac{1}{2}, \\
         O\left( \begin{matrix}
                    |\l|^{2\a} & |\l|^{2\a}\\
                    |\l|^{2\a} & |\l|^{2\a}
         \end{matrix}\right), & \textrm{if } -\frac{1}{2} < \alpha \leq 0,
\end{cases}  \qquad \textrm{as } \l \to 0
\end{equation}
and
\begin{equation}
\Psi^{(3)}_\a(\l)= O\left(\frac{1}{\l^2}\right), \qquad \textrm{as } \l \to \infty.
\end{equation}
From the above two formulas, one can see that each entry of $\Psi^{(3)}_\a(\l)$ has an integrable singularity at $\l = 0$ and $\l = \infty$. Thus, recalling the fact that $\Psi^{(3)}_\a(\l)$ is analytic in the upper half plane, we have $ \D\int_{\mathbb{R}} \Psi^{(3)}_{\a,+}(\l) d \l =0$ from Cauchy's theorem, which means
\begin{equation}\label{Cau-for-Psi2}
\int_{\mathbb{R}} \Psi^{(2)}_{\a,+}(\l) \Psi^{(2)}_{\a,-}(\l)^{*} d \l =0.
\end{equation}
Adding its Hermitian conjugate to the above formula, we obtain
\begin{equation}
\int_{\mathbb{R}} \left[ \Psi^{(2)}_{\a,+}(\l) \Psi^{(2)}_{\a,-}(\l)^{*} + \Psi^{(2)}_{\a,-}(\l) \Psi^{(2)}_{\a,+}(\l)^{*} \right] d\l =0.
\end{equation}
With the jump matrices in \eqref{Psi2-J-neg} and \eqref{Psi2-J-pos}, we have
\begin{equation}
\int_{\mathbb{R}} \left[ \Psi^{(2)}_{\a,-}(\l)\left( \begin{matrix}
                    2(1-s_1 s_3) & 0\\
                    0 & 2
         \end{matrix}\right)\Psi^{(2)}_{\a,-}(\l)^{*}\right] d\l =0,
\end{equation}
where the relations $s_3=\bar{s}_1$ and $|s_1|<1$ are also used; see \eqref{stokes-real-im} and \eqref{real-alpha-k}. A straightforward consequence of the above formula is that $\Psi^{(2)}_{\a,-}(\l)$ is identically zero. Furthermore, it follows from \eqref{Psi2-J-neg} and \eqref{Psi2-J-pos} that $\Psi^{(2)}_{\a, +}(\l)$ is also identically zero. Therefore, $\Psi^{(2)}_\a(\l)$ vanishes for all $\l \in \mathbb{C}$.

Then, the relation \eqref{vanishing-psi2-psi1} between $\Psi^{(1)}_\a(\l)$ and $\Psi^{(2)}_\a(\l)$ implies that $\Psi^{(1)}_\a(\l) \equiv 0 $ for all $\l \in \mathbb{C}$. This completes the proof of the vanishing lemma.
\end{proof}

\section{Nonlinear steepest descent analysis}\label{sec:thm2-bc:thm3}

The nonlinear steepest descent for RH problems is a powerful asymptotic method, which was introduced by Deift and Zhou \cite{Deift1993} to study the asymptotics of the MKdV equation. Later, this method has been successfully developed to solve asymptotic problems related to orthogonal polynomials, random matrices, Painlev\'e equations, as well as other integrable PDEs; for example, see \cite{Buck:Miller2015,Cla:Tamara2010,Claeys2008,Dei:Kri:McL:Ven:Zhou1999-2,Deift1995,Its2003,Kri:McL1999,Wong:Zhao2016,Xu:Dai:Zhao2014}.

In this section, we will apply this method to derive the asymptotics of PII as $x \to \pm \infty$. From the formula \eqref{stokes-real-im}, the associated Stokes multiplier $s_2$ equals 0 for both the real and purely imaginary Ablowitz-Segur solutions. Therefore, we may conduct the steepest descent analysis for both cases at the same time.

Let us consider the following RH problem for $\Psi_\a(\l)$, which is equivalent to the RH problems in \cite[Chap. 11]{Fokas2006} and \cite{Its2003} with $s_2 = 0$.

\begin{itemize}
\item[(a)] $\Psi_\a(\l)$ is analytic for $\l \in \mathbb{C} \cut \Sigma$, where $\Sigma =\D\cup_{k= 1,3,4,6} \gamma_k$, with $\gamma_k=\{\l \in \mathbb{C}: \textrm{arg}\, \l = \frac{\pi}{6}+\frac{\pi}{3}(k-1)\}$ oriented to infinity; see Figure \ref{contour(s2=0)};

\item[(b)] $\Psi_{\a,+}(\l)=\Psi_{\a,-}(\l) S_k$ for $\l \in \gamma_k$, where $S_k$'s are the upper or lower triangular matrices given in \eqref{Psi-jumps};

\begin{figure}[h]
  \centering
  \includegraphics[width=8cm]{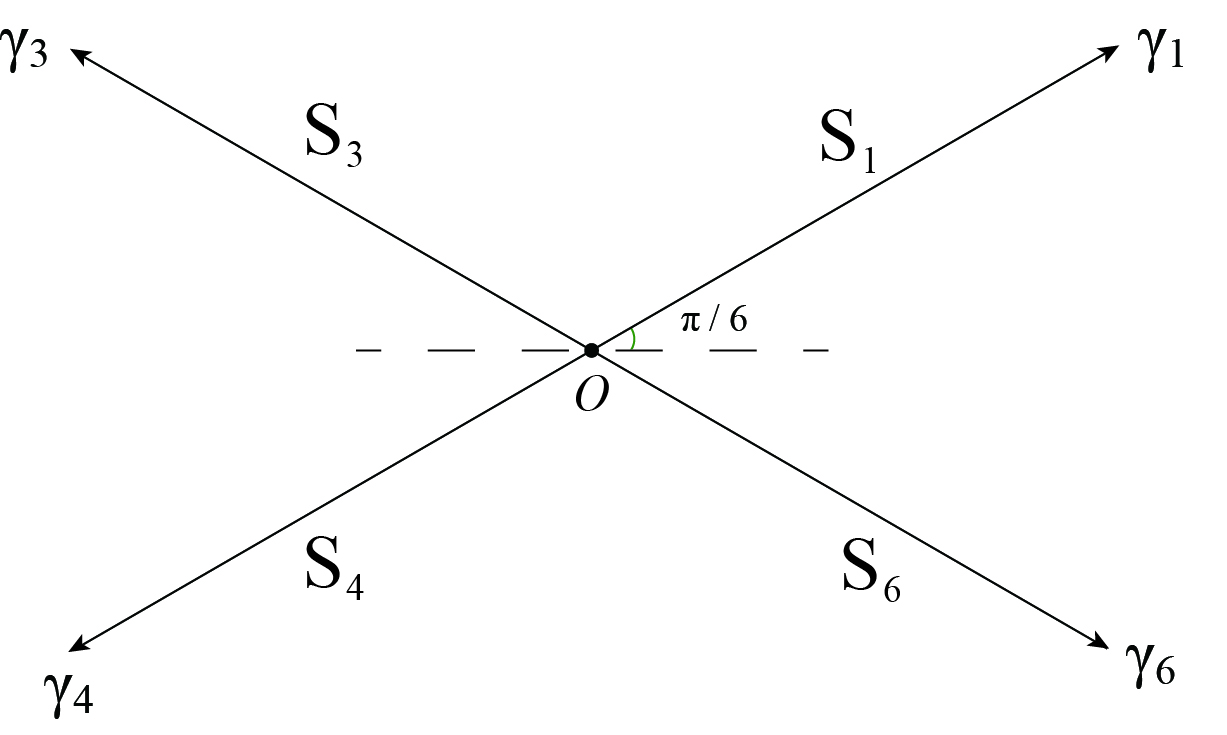}
  \caption{The contour $\Sigma$ for the RH problem of $\Psi_\a(\l)$  with $s_2=0$, together with the associated jump matrices on corresponding rays.}\label{contour(s2=0)}
\end{figure}

 \item[(c)] As $\l \to \infty$, $\Psi_{\a}(\l)$ has the following asymptotic behavior:
 \begin{equation} \label{psia-infinity}
  \Psi_\a(\l) = \biggl[ I+ \frac{\Psi_{-1}(x)}{\l}+O(\l^{-2}) \biggr]e^{-\theta (\l)\sigma_3}  \quad \textrm{with } \theta(\l) = i(\frac{4}{3}\l^3 + x\l);
 \end{equation}

 \item[(d)] At $\l =0$, $\Psi_\a(\l)$ has the following behaviors:
 \begin{equation}
  \Psi_\a(\l) = O\left(\begin{matrix}
  |\l|^{-\a} &  |\l|^{-\a}\\
   |\l|^{-\a} & |\l|^{-\a}
  \end{matrix}\right), \qquad \textrm{if } 0<  \re \alpha < \frac{1}{2}
\end{equation}
and
\begin{equation}
  \Psi_\a(\l) =
  O\left(\begin{matrix}
  |\l|^{\a} &  |\l|^{\a}\\
   |\l|^{\a} & |\l|^{\a}
  \end{matrix}\right), \qquad \textrm{if } -\frac{1}{2} < \re \alpha \leq 0.
\end{equation}
\end{itemize}

As $x\to +\infty$, the asymptotics for the solutions $u_{\textrm{AS}}(x;\a)$ and $u_{i\textrm{AS}}(x;\alpha)$ can be derived from the following asymptotic results in Its and Kapaev \cite[Thm. 2.2]{Its2003}:
\begin{equation} \label{its-kapa-for}
  u(x;\a) = u_1(x; \a)  - \frac{i s_3}{2 \sqrt{\pi} x^{1/4}} e^{-\frac{2}{3} x^{3/2}} [1+O(x^{-3/4})] + O(s_3^2 x^{-7/4}e^{-\frac{4}{3} x^{3/2}}), \quad x \to +\infty,
\end{equation}
where $u_1(x; \a) \sim \a / x$. Recalling the value of the Stokes multiplier $s_3$ in \eqref{stokes-real-im} 
we can obtain  \eqref{asy-pos} and \eqref{asy-pos-im} from the above formula. Here, one should be careful to note that the solution $u_1(x; \a)$ is the solution of PII corresponding to the Stokes multipliers
\begin{equation}
  s_1 = -2 \sin (\pi \a) \quad \textrm{and} \quad s_2=s_3 = 0.
\end{equation}
This solution is neither real nor purely imaginary for real $x$ according to Theorems \ref{real-reduction-thm} and \ref{im-reduction-thm} in the previous section. Since $u_{\textrm{AS}}(x;\a)$ is a real solution when $s_3 = -\sin(\pi \a) + ik$ with $\a \in (-\frac{1}{2}, \frac{1}{2})$ and $ k \in (-\cos(\pi \a),  \cos(\pi \a))$, one needs to take the real part of \eqref{its-kapa-for} to get the asymptotics of $u_{\textrm{AS}}(x;\a)$ in \eqref{asy-pos}. Similarly, one needs to take the imaginary part of \eqref{its-kapa-for} to obtain \eqref{asy-pos-im} for $u_{i\textrm{AS}}(x;\a)$ whose Stokes multiplier $s_3 = -\sin(\pi \a) + ik$ with $\a \in i\mathbb{R}$ and $ k \in i\mathbb{R}$.

Next, in the remaining sections of this paper, we will focus on the nonlinear steepest analysis for the above RH problem for $\Psi_\a(\l)$ as $x\to - \infty$.

\subsection{Normalization}

Since we will let $x \to -\infty,$ it is natural to consider the following scaling
\begin{equation}
\l=(-x)^{1/2}z,
\end{equation}
such that the phase function $\theta(\l) $ in \eqref{psia-infinity} becomes
\begin{equation} \label{tilde-theta-def}
\theta(\l) = t\tilde{\theta}(z), \quad \textrm{with } \tilde{\theta}(z) := i(\frac{4}{3}z^3 -z) \quad \textrm{and} \quad t=(-x)^{3/2}.
\end{equation}
Then, the first transformation is defined as
\begin{equation} \label{Psi-U}
U(z) = \Psi_\a(\l(z))\exp(t \tilde{\theta}(z)\sigma_3),
\end{equation}
where the behavior of $U(z)$ at infinity is normalized, i.e., $U(z) = I + O(1/z)$ as $z \to \infty$.

Before we state the RH problem for $U(z)$, we note that the above transformation will change the jumps $S_k$ in \eqref{Psi-jumps} to be $e^{-t\tilde{\theta}(z) \sigma_3}  S_k e^{t\tilde{\theta}(z) \sigma_3} $. This keeps the diagonal entries unchanged, but multiplies the upper and lower triangular entry by $e^{\mp 2t \tilde\theta (z)}$, respectively. Then, it is important for us to know the properties of $\re \tilde\theta (z) $ in the complex-$z$ plane. From \eqref{tilde-theta-def}, it is easily seen that $z_{\pm}=\pm \frac{1}{2}$ are the stationary points for $\tilde\theta (z) $. Moreover, we get the signature properties of $\re \tilde{\theta}(z)$ as shown in Figure \ref{sign-theta}.

\begin{figure}[h]
  \centering
  \includegraphics[width=12cm]{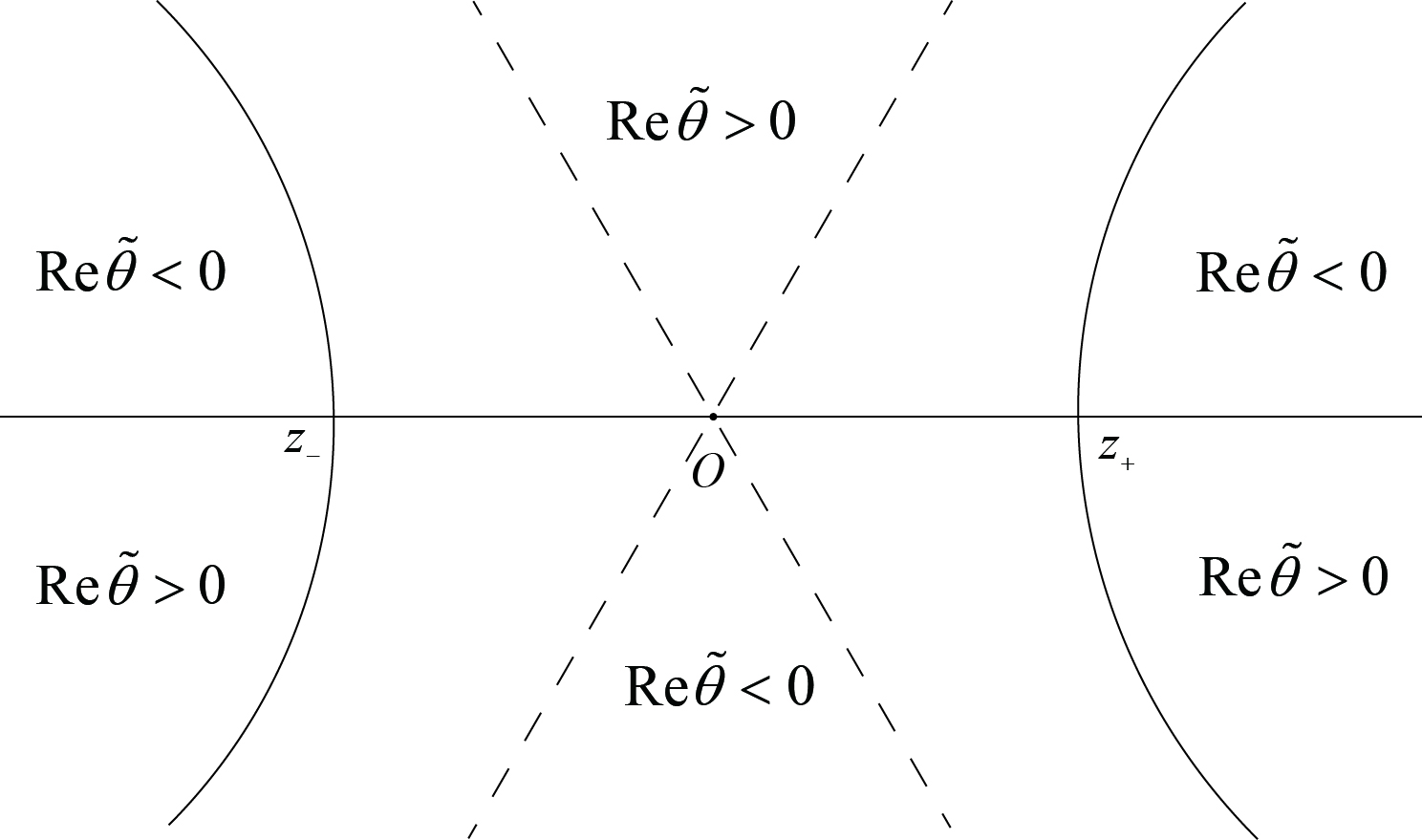}
  \caption{The signature properties of $\re \tilde{\theta}(z)$, where the dashed lines are the rays  $ \{z \in \mathbb{C}: \textrm{arg}\, z = \frac{k\pi}{3}, k = 1,2,4,5\}$.   }\label{sign-theta}
\end{figure}

Note that, in each sector bounded by the rays $\gamma_k$ in Figure \ref{contour(s2=0)}, $\Psi_\a(\l)$ is indeed an analytic function with only a branch point at $\l = 0$. As $z$ is only a rescaling of $\l$, it is possible for us to deform the original contour $\Sigma$ such that the new one
is in accordance with the signature table of $\re \tilde{\theta}(z)$ in Figure \ref{sign-theta}. Finally, based on \eqref{Psi-U} and the RH problem for $\Psi_\a(\l)$,  $U(z)$ satisfies a RH problem as follows:
\begin{itemize}
\item[(a)] $U(z)$ is analytic for $ z \in \mathbb{C} \cut \Sigma_U$;

\begin{figure}[h]
  \centering
  \includegraphics[width=14cm]{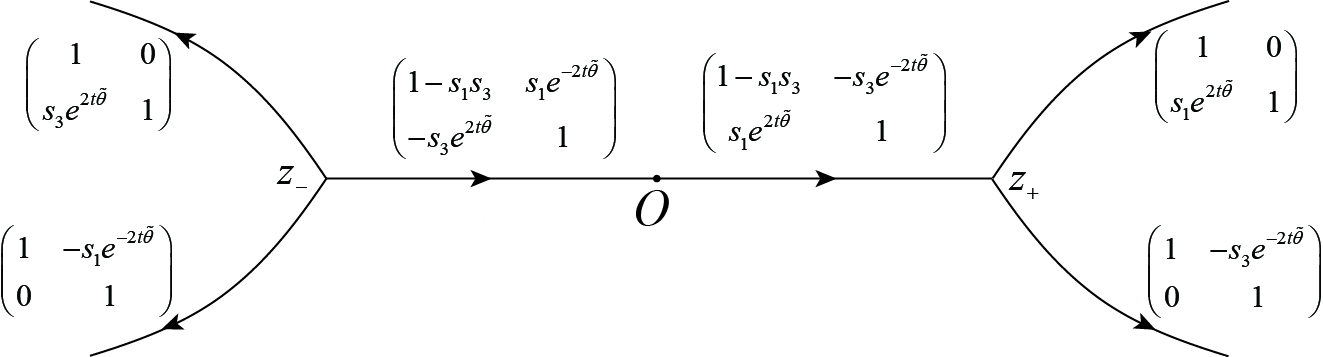}
  \caption{The contour $\Sigma_U$ and associated jump matrices for the RH problem of $U(z)$. Note that the contour emanating from $z_{\pm}$ and going to infinity is chosen to agree with the property of $\re\tilde{\theta}(z)$. Then, the off-diagonal entries in the corresponding jump matrices are exponentially small when $z$ is bounded away from $z_{\pm}$ as $t\to +\infty$.} \label{contour(sd)}
\end{figure}

\item[(b)] $U_{+}(z)=U_{-}(z) J_{U}(z)$ for $z \in \Sigma_U$, where the contour $\Sigma_U$ and the jump $J_U(z)$ are depicted in Figure \ref{contour(sd)};

 \item[(c)] As $z \to \infty$, $U(z) = I + O(1/z) $;

 \item[(d)] At $z =0$, $U(z)$ has the following behaviors:
 \begin{equation}\label{U-origin-asy}
  U(z) = O\left(\begin{matrix}
  |z|^{-\a} &  |z|^{-\a}\\
   |z|^{-\a} & |z|^{-\a}
  \end{matrix}\right), \qquad \textrm{if } 0<  \re \alpha < \frac{1}{2}
\end{equation}
and
\begin{equation}\label{U-origin-asy-alpha-neg}
 U(z) =
  O\left(\begin{matrix}
  |z|^{\a} &  |z|^{\a}\\
   |z|^{\a} & |z|^{\a}
  \end{matrix}\right), \qquad \textrm{if } -\frac{1}{2} < \re \alpha \leq 0;
\end{equation}

 \item[(e)] As $z \to  z_\pm$, $U(z)$ is bounded.

\end{itemize}

\subsection{Contour deformation}

From Figures \ref{sign-theta} and \ref{contour(sd)}, one can see that the jumps $J_U(z)$ are exponentially close to the identity matrix when $t \to +\infty$, except the one on the line segment $[z_-, z_+]$. Indeed, since $\re \tilde{\theta}(z) = 0$ for $z \in [z_-, z_+]$, the jump matrices
\begin{equation}
\left(
  \begin{array}{cc}
  1-s_1 s_3 & -s_3 e^{-2t \tilde \theta}\\
    s_1 e^{2t \tilde \theta} & 1
  \end{array}
    \right) \quad \textrm{for } z \in (0,z_+), \qquad
  \left(
  \begin{array}{cc}
  1-s_1 s_3 & s_1 e^{-2t \tilde \theta}\\
    -s_3 e^{2t \tilde \theta} & 1
  \end{array}
    \right) \quad \textrm{for } z \in (z_-, 0)
\end{equation}
are highly oscillating as $t \to +\infty$. To remove the oscillatory off-diagonal entries, we notice that the above matrices satisfies the following
LDU-decomposition:
\begin{eqnarray}
\left(
  \begin{array}{cc}
  1-s_1 s_3 & -s_3 e^{-2t \tilde \theta}\\
    s_1 e^{2t \tilde \theta} & 1
  \end{array}
    \right) &=&  \left(
  \begin{array}{cc}
  1 & 0\\
    \frac{s_1 e^{2t\tilde{\theta}(z)}}{1-s_1 s_3} & 1
  \end{array}
    \right)  \left(
  \begin{array}{cc}
  1-s_1 s_3 & 0\\
    0 & \frac{1}{1-s_1 s_3}
  \end{array}
    \right) \left(
  \begin{array}{cc}
  1 & -\frac{s_3 e^{-2t\tilde{\theta}(z)}}{1-s_1 s_3}\\
    0 & 1
  \end{array}
    \right) \nonumber\\
  &:=&S_{L_1}S_{D}S_{U_1},
\end{eqnarray}
and
\begin{eqnarray}
\left(
  \begin{array}{cc}
  1-s_1 s_3 & s_1 e^{-2t \tilde \theta}\\
    -s_3 e^{2t \tilde \theta} & 1
  \end{array}
    \right)&=& \left(
  \begin{array}{cc}
  1 & 0\\
    -\frac{s_3 e^{2t\tilde{\theta}(z)}}{1-s_1 s_3} & 1
  \end{array}
    \right)  \left(
  \begin{array}{cc}
  1-s_1 s_3 & 0\\
    0 & \frac{1}{1-s_1 s_3}
  \end{array}
    \right) \left(
  \begin{array}{cc}
  1 & \frac{s_1 e^{-2t\tilde{\theta}(z)}}{1-s_1 s_3}\\
    0 & 1
  \end{array}
    \right)\nonumber \\
    &:=&S_{L_2}S_{D}S_{U_2}.
\end{eqnarray}
\begin{figure}[h]
  \centering
  \includegraphics[width=14cm]{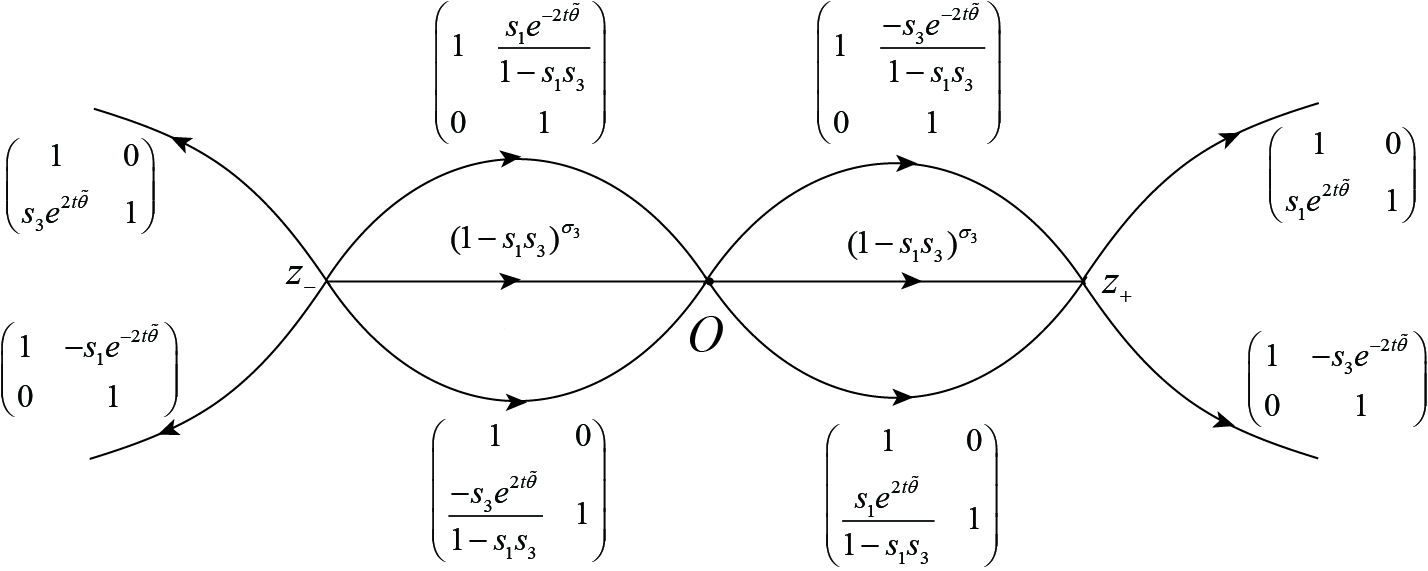}
  \caption{The contour $\Sigma_T$ and associated jump matrices for the RH problem of $T(z)$. Similar to Figure \ref{contour(sd)}, due to the property of $\re\tilde{\theta}(z)$, all the off-diagonal entries in the jump matrices are exponentially small when $z$ is bounded away from $z_{\pm}$ and the origin as $t\to +\infty$.} \label{contour(LD)}
\end{figure}
From the above factorization, we define the second transformation from $U \to T$ as follows:
\begin{equation}\label{def-T}
T(z) := \begin{cases}\D
U(z), & \textrm{for $z$ outside the two lens shaped regions},\\
U(z)S^{-1}_{U_1}, & \textrm{for $z$ in the upper part of the right lens shaped region},\\
U(z)S_{L_1}, & \textrm{for $z$ in the lower part of the right lens shaped region},\\
U(z)S^{-1}_{U_2}, & \textrm{for $z$ in the upper part of the left lens shaped region},\\
U(z)S_{L_2}, & \textrm{for $z$ in the lower part of the left lens shaped region};
\end{cases}
\end{equation}
see Figure \ref{contour(LD)}.
Then, $T(z)$ solves the following RH problem:

\begin{itemize}
\item[(a)] $T(z)$ is analytic for $ z \in \mathbb{C} \cut \Sigma_T$;

\item[(b)] $T_{+}(z)=T_{-}(z) J_{T}(z)$ for $z \in \Sigma_T$, where the contour $\Sigma_T$ and the jump $J_T(z)$ are depicted in Figure \ref{contour(LD)};

 \item[(c)] As $z \to \infty$, $T(z) = I + O(1/z) $;

 \item[(d)] At $z =0$, $T(z)$ has the same behaviors as $U(z)$ in \eqref{U-origin-asy} and \eqref{U-origin-asy-alpha-neg};

 \item[(e)] As $z \to z_\pm$, $T(z)$ is bounded.
\end{itemize}

\subsection{Global parametrix}

From Figures \ref{sign-theta} and \ref{contour(LD)}, we see that the jump matrix for $T$ is of the form $I$ plus an exponentially
small term for $z $ bounded away from $[z_-, z_+]$. Neglecting the exponential small terms, we arrive at
an approximating RH problem for $N(z)$ as follows:

\begin{itemize}
  \item[(a)] $N(z)$ is analytic for $z \in \mathbb{C} \setminus [z_-, z_+]$, where $z_\pm = \pm \frac{1}{2}$;

  \item[(b)] On the line segment $[z_-, z_+]$, we have
\begin{align}
N_{+}(z)=N_{-}(z) S_{D}, \qquad \textrm{with } S_D = (1-s_1 s_3)^{\sigma_3};
\end{align}

\item[(c)]  $N(z)= I+O(1/z) $, \qquad as $z\to \infty$.

\end{itemize}

A solution to the above RH problem can be constructed explicitly,
\begin{equation} \label{parametrix-global}
  N(z)=\left(\frac{z+1/2}{z-1/2}\right)^{\nu \sigma_3}, \qquad \textrm{with } \nu= -\frac{1}{2\pi i} \ln(1-s_1 s_3).
\end{equation}
Note that, with our choices of the Stokes multiplies $s_1$ and $s_3$ in \eqref{stokes-real-im} with either conditions \eqref{real-alpha-k} or \eqref{im-alpha-k}, the above constant $\nu$ is always purely imaginary. Moreover, the branch cut of $\left(\frac{z+1/2}{z-1/2}\right)^{\nu}$ is chosen along the segment $[-\frac{1}{2}, \frac{1}{2}]$, with $\arg (z\pm \frac{1}{2}) \in (-\pi, \pi)$, such that $
\left(\frac{z+1/2}{z-1/2}\right)^{\nu} \to 1$ as $ z\to \infty$.

The jump matrices of $T(z) [N(z) ]^{-1}$ are not uniformly close to the unit matrix near the end-points $z_\pm$ and
0, thus local parametrices have to be constructed in neighborhoods of these points.

\subsection{Local parametrices near $z = \pm \frac{1}{2}$ } \label{sec:local-z+}

Let us consider the right endpoint $z_+ = \frac{1}{2}$ first. Let $U(z_+, \delta)$ be a small disk with center at $z_+$ and radius $\delta > 0$. We seek a $2 \times 2$ matrix-valued
function $P^{(r)}(z)$ in $U(z_+, \delta)$, such that it satisfies the same jumps as $T$, possesses the same behavior as $T$ near $z_+$, and matches with $N(z)$ on the boundary $\partial U(z_+, \delta)$ of the disk. Therefore, $P^{(r)}(z)$ satisfies a RH problem as follows:
\begin{itemize}
  \item[(a)] $P^{(r)}(z)$ is analytic in $U(z_+, \delta) \setminus \Sigma_T$;

  \item[(b)] In $U(z_+, \delta)$, $P^{(r)}(z)$ satisfies the same jump conditions as $T(z)$ does; see Figure \ref{contour(psi-r)}, where we enlarge the contour for better illustration;

  \begin{figure}[h]
  \centering
  \includegraphics[width=7cm]{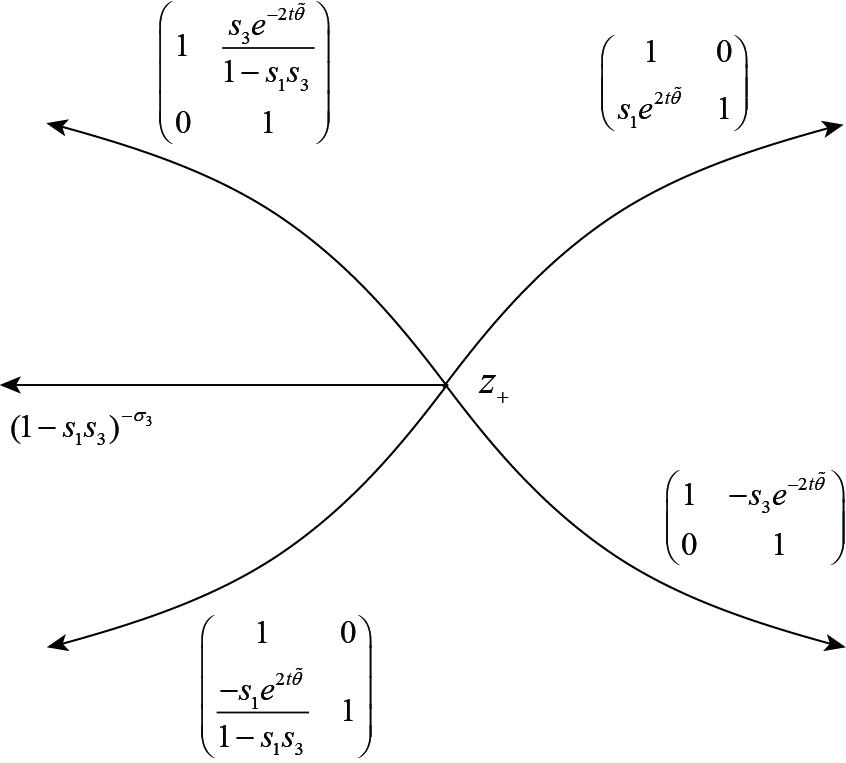}
  \caption{The contour  $\Sigma_T $ in $U(z_{+}, \delta)$ and associated jump matrices for the RH problem of $P^{(r)}(\zeta)$.}\label{contour(psi-r)}
\end{figure}
  \item[(c)] $P^{(r)}(z)$ fulfils the following matching condition on $\partial U(z_+, \delta) = \{ z : \ |z - z_+| = \delta \}$:
  \begin{equation} \label{psir-psid-match}
    P^{(r)}(z) [N(z) ]^{-1} = I + O (t^{-1/2}) \qquad \textrm{as } t \to + \infty;
  \end{equation}
  \item[(d)] As $z \to z_+$, $P^{(r)}(z)$ is bounded.
\end{itemize}

The solution $P^{(r)}(z)$ for the above RH problem is given in terms of the parabolic cylinder functions $D_{\nu}(\zeta)$; see Fokas et al. \cite[Sec. 9.4]{Fokas2006}. 
To give the exact formulas for $P^{(r)}(z)$, let us define the function $\zeta(z)$ as follows:
\begin{align} \label{zeta-defn}
\zeta(z):=2\sqrt{-\tilde{\theta}(z)+\tilde{\theta}(z_+)} =  2  \, \sqrt{-\frac{4i}{3} z^3+ iz -\frac{i}{3}},
\end{align}
where the branch of the square root is taken such that $\arg{(z-\frac{1}{2})} \in (-\pi, \pi)$. This function gives us a conformal mapping in the neighbourhood of $z=z_+ = 1/2$. Moreover, we have
\begin{equation} \label{zeta-z+-nei}
\zeta(z) \sim e^{3i\pi/4} 2\sqrt{2} \, (z-\frac{1}{2}), \qquad \textrm{as } z \to \frac{1}{2}.
\end{equation}
Then, the solution $P^{(r)}(z)$ is given as (see  Fokas et al. \cite[p. 322-325]{Fokas2006})
\begin{equation}
  P^{(r)}(z)= \beta(z)^{\sigma_3} \left(\frac{-h_1}{s_3}\right)^{-\sigma_3 /2} e^{it\sigma_3 /3} 2^{-\sigma_3 /2}\left(\begin{matrix}
\sqrt{t}\,\zeta(z) & 1\\
1 & 0
\end{matrix}\right) Z(\sqrt{t}\,\zeta(z)) e^{t\tilde{\theta}(z)\sigma_3}  \left(\frac{-h_1}{s_3}\right)^{\sigma_3 /2},
\end{equation}
where $\beta(z)$ is holomorphic in the neighborhood of $z=z_+=\frac{1}{2}$ with
\begin{align} \label{beta-defn}
\beta(z):=\left(\sqrt{t}\,\zeta(z)\frac{z+1/2}{z-1/2}\right)^{\nu }, \quad  \quad  \beta(z_+) = (8t)^{\nu /2}e^{3i\pi \nu /4}.
\end{align}
Here, the function $Z(\zeta)$ is defined in terms of the parabolic cylinder functions
\begin{equation} \label{zrh-def}
Z(\zeta):= \begin{cases}\D
Z_0(\zeta), &\arg \zeta \in (-\frac{\pi}{4}, 0),\\
Z_k(\zeta), &\arg \zeta \in (\frac{k-1}{2} \pi, \frac{k}{2}\pi), \ k = 1,2,3,\\
Z_4(\zeta), &\arg \zeta \in (\frac{3\pi}{2}, \frac{7\pi}{4}),
\end{cases}
\end{equation}
with
\begin{equation} \label{z0-def}
Z_0(\zeta)= 2^{-\sigma_3/2} \left(
  \begin{matrix}
  D_{-\nu-1}(i\zeta) & D_{\nu}(\zeta)\\
   \frac{d}{d\zeta}D_{-\nu-1}(i\zeta) & \frac{d}{d\zeta}D_{\nu}(\zeta)
  \end{matrix}
    \right) \left(
  \begin{matrix}
  e^{\frac{i\pi}{2}(\nu + 1)} & 0\\
    0 & 1
  \end{matrix}
    \right),
\end{equation}
and
\begin{eqnarray}
Z_{k+1}(\zeta) = Z_k (\zeta)H_k, \qquad k = 0,1,2,3.
\end{eqnarray}
The constant matrices $H_j$ above are upper or lower triangular ones given as follows
\begin{equation} \label{def-h-0123}
  H_0= \left(
  \begin{matrix}
  1 & 0\\
    h_0 & 1
  \end{matrix}
    \right), \quad H_1= \left(
  \begin{matrix}
 1 & h_1\\
    0 & 1
  \end{matrix}
    \right), \quad H_{k+2}=e^{i\pi (\nu+\frac{1}{2})\sigma_3}H_k e^{-i\pi (\nu+\frac{1}{2})\sigma_3}, \quad k = 0,1,
\end{equation}
where the constants are
\begin{equation}\label{def-h-01}
  h_0 = -i\frac{\sqrt{2\pi}}{\Gamma(\nu+1)} \qquad \textrm{and} \qquad h_1 = \frac{\sqrt{2\pi}}{\Gamma(-\nu)}e^{i \pi \nu}.
\end{equation}
Here $\nu$ is the same as that in \eqref{parametrix-global}. Note that, from the formula (9.4.23) in \cite{Fokas2006}, we have
\begin{equation} \label{prz-large-t}
P^{(r)}(z)= \left(\begin{matrix}
 1 & -\frac{\nu s_3}{h_1}e^{\frac{2it}{3}} \beta^2(z)\frac{1}{\sqrt{t}\zeta}\\
-\frac{h_1}{s_3}e^{-\frac{2it}{3}} \beta^{-2}(z)\frac{1}{\sqrt{t}\zeta} & 1
\end{matrix} \right)\times (I+O(t^{-1}))N(z), \qquad t \to +\infty
\end{equation}
uniformly for $z \in \partial U(z_+, \delta) $.

Near the other endpoint $z=z_{-} = - 1/2$, the parametrix $P^{(l)}(z)$ can be constructed in a similar way. Indeed, from the symmetry of the RH problem for $T(z)$ in Figure \ref{contour(LD)}, we may simply define
\begin{equation} \label{prl-realtion}
P^{(l)}(z)=\sigma_2 P^{(r)}(-z) \sigma_2.
\end{equation}

\subsection{Local parametrix near the origin}

Similar to the previous section, let $U(0, \delta)$ be a small disk with center at $0$ and radius $\delta > 0$. We seek a $2 \times 2$ matrix-valued function $P^{(0)} (z)$  in $U(0, \delta)$, such that it satisfies the following RH problem:
\begin{itemize}
  \item[(a)] $P^{(0)}(z)$ is analytic in $U(0, \delta) \setminus \Sigma_T$;

  \item[(b)] In $U(0, \delta)$, $P^{(0)}(z)$ satisfies the same jump conditions as $T(z)$ does; see Figure \ref{contour(LD)};

  \item[(c)] $P^{(0)}(z)$ fulfils the following matching condition on $\partial U(0, \delta) = \{ z : \ |z| = \delta \}$:
  \begin{equation} \label{psi0-psid-match}
    P^{(0)}(z) [N(z) ]^{-1} = I + O (t^{-1}) \qquad \textrm{as } t \to + \infty;
  \end{equation}

  \item[(d)] At $z =0$, $P^{(0)}(z)$ has the same behaviors as $U(z)$ in \eqref{U-origin-asy} and \eqref{U-origin-asy-alpha-neg}.
\end{itemize}

Unlike the previous section, we cannot construct the solution to the above RH problem explicitly. Luckily, as $x \to -\infty$, we only need the leading term for the asymptotics of $u(x; \a)$ (cf. \eqref{asy-neg} and \eqref{asy-neg-im}), which is also enough for us to establish the connection formulas. Moreover, while matching the local and global parametrices, one can see that the error terms from the endpoints $z_\pm$ are dominant comparing with that from the origin; see \eqref{psir-psid-match} and \eqref{psi0-psid-match}. Therefore, it is enough for us to prove the existence of $P^{(0)}(z)$ which satisfies the above RH problem. Note that, similar treatment is also adopted by Kriecherbauer and McLaughlin in \cite{Kri:McL1999} to handle the asymptotic behaviors of polynomials orthogonal with respect to Freud weights. Due to the lack of analyticity in the neighbourhood of the origin, a model RH problem is constructed in the neighbourhood of the origin, where the explicit expression of the solution is not given. Recently, Wong and Zhao \cite{Wong:Zhao2016} obtain more information about the solution by showing that its entries satisfy certain integral equations.

To prove the existence of $P^{(0)}(z)$, let us first consider the following model RH problem for $L(\eta)$.

\begin{itemize}

\item[(a)] $L(\eta)$ is analytic in $\mathbb{C} \cut \Gamma_L$;
\begin{figure}[h]
  \centering
  \includegraphics[width=8cm]{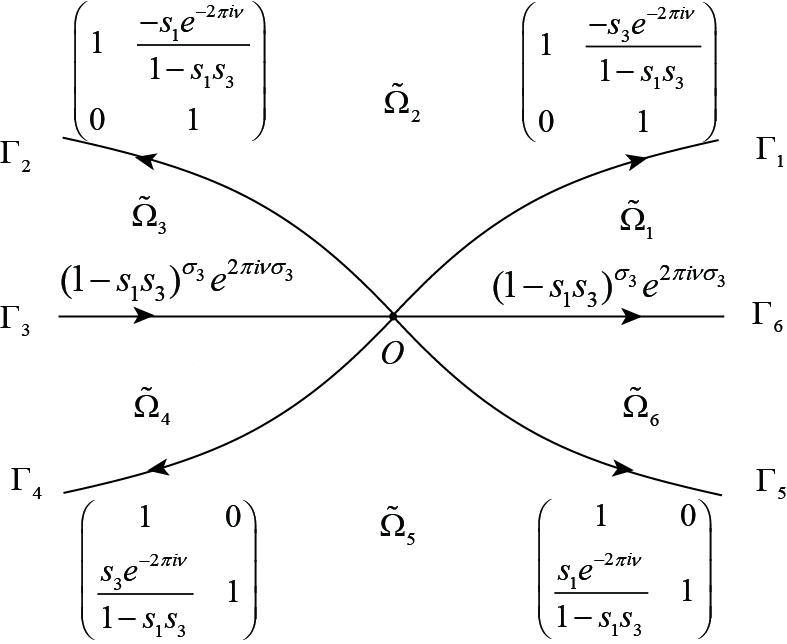}
  \caption{The contour $\Sigma_L$ and associated jump matrices for the RH problem of $L(\eta)$.}\label{contour(origin)}
\end{figure}

\item[(b)] $L(\eta)$ satisfies the following jump relations on $\Gamma_L$:
\begin{equation}\label{L+ L-}
 L_{+}(\eta) = L_{-}(\eta) J_L,
\end{equation}
where the contour $\Gamma_L$ and the jump $J_L$ are given in Figure \ref{contour(origin)};

\item[(c)] At $\eta = 0$, $L(\eta)$ has the following behaviors:
 \begin{equation}\label{L-origin-asy}
  L(\eta) = O\left(\begin{matrix}
  |\eta|^{-\a} &  |\eta|^{-\a}\\
   |\eta|^{-\a} & |\eta|^{-\a}
  \end{matrix}\right), \quad \arg \eta \in (-\pi, \pi), \quad \textrm{if } 0<  \re \alpha < \frac{1}{2}
\end{equation}
and
\begin{equation}\label{L-origin-asy-alpha-neg}
  L(\eta) =
  O\left(\begin{matrix}
  |\eta|^{\a} &  |\eta|^{\a}\\
   |\eta|^{\a} & |\eta|^{\a}
  \end{matrix}\right) , \quad \arg \eta \in (-\pi, \pi), \quad \textrm{if } -\frac{1}{2} < \re \alpha \leq 0;
\end{equation}

\item[(d)] $L(\eta)$ has the following behavior at infinity:
\begin{equation}
L(\eta) = \left(I+O\left(\frac{1}{\eta}\right)\right)e^{i\eta\sigma_3} , \qquad  \eta\to\infty.
\end{equation}

\end{itemize}

By standard arguments based on Liouville's theorem, it can be easily verified that if the solution of the above RH problem exists, then it is unique. Next, we proceed to justify the existence of the solution by proving the following vanishing lemma for the RH problem.

\begin{lemma}
For $\re \a \in (-\frac{1}{2}, \frac{1}{2})$, suppose that $L^{(1)}(\eta)$ satisfies the same jump conditions \eqref{L+ L-} and the same boundary conditions \eqref{L-origin-asy} and \eqref{L-origin-asy-alpha-neg} as $L(\eta)$, with just the behavior at infinity being altered to
\begin{equation}
L^{(1)}(\eta) = O\left(\frac{1}{\eta}\right)e^{i\eta\sigma_3}, \qquad \textrm{as} \quad \eta\to\infty.
\end{equation}
Then $L^{(1)}(\eta)$ is trivial, that is, $L^{(1)}(\eta) \equiv 0$.
\end{lemma}

\begin{proof}
The proof is similar to that of Lemma \ref{lemma-vanishing-1}. Let us introduce the function $L^{(2)}(\eta)$ as follows:
\begin{equation}
L^{(2)}(\eta) :=  \begin{cases}\D
 L^{(1)}(\eta)e^{-i\eta\sigma_3}, &\quad \eta \in \tilde{\Omega}_2 \cup \tilde{\Omega}_5,\\
 L^{(1)}(\eta)J_1e^{-i\eta\sigma_3}, &\quad \eta \in \tilde{\Omega}_1,\\
 L^{(1)}(\eta)J_{2}^{-1}e^{-i\eta\sigma_3}, &\quad \eta \in \tilde{\Omega}_3, \\
 L^{(1)}(\eta)J_4e^{-i\eta\sigma_3}, &\quad \eta \in \tilde{\Omega}_4,\\
 L^{(1)}(\eta)J_{5}^{-1}e^{-i\eta\sigma_3}, &\quad \eta \in \tilde{\Omega}_6,
\end{cases}
\end{equation}
with $J_k$ the jump matrix on $\Gamma_k$; see Figure \ref{contour(origin)}.
It is easy to verify that $L^{(2)}(\eta)$ solves the following RH problem:
\begin{itemize}
\item[(a)] $L^{(2)}(\eta)$ is analytic in $\mathbb{C} \cut \mathbb{R}$;
\item[(b)] $L^{(2)}(\eta)$ satisfies the following jump relations on $\mathbb{R}\cut \{0\}$,
\begin{eqnarray}
L^{(2)}_{+}(\eta)=L^{(2)}_{-}(\eta) e^{i(\eta+\pi\nu)\sigma_3}\left(\begin{matrix}
1-s_1 s_3 & s_1\\
-s_3 & 1
\end{matrix}\right)e^{-i(\eta-\pi\nu)\sigma_3}, \quad \eta\in (-\infty, 0),\label{H-J-neg}\\
L^{(2)}_{+}(\eta)=L^{(2)}_{-}(\eta) e^{i(\eta+\pi\nu)\sigma_3}\left(\begin{matrix}
1-s_1 s_3 & -s_3\\
s_1 & 1
\end{matrix}\right)e^{-i(\eta-\pi\nu)\sigma_3}, \quad \eta\in (0, \infty)\label{H-J-pos};
\end{eqnarray}
\item[(c)] $L^{(2)}(\eta)$ has the following behavior near the origin.
 \begin{equation}
  L^{(2)}(\eta) = O\left(\begin{matrix}
  |\eta|^{-\a} &  |\eta|^{-\a}\\
   |\eta|^{-\a} & |\eta|^{-\a}
  \end{matrix}\right), \qquad \textrm{if } 0<  \re \alpha < \frac{1}{2}
\end{equation}
and
\begin{equation}
 L^{(2)}(\eta) =
  O\left(\begin{matrix}
  |\eta|^{\a} &  |\eta|^{\a}\\
   |\eta|^{\a} & |\eta|^{\a}
  \end{matrix}\right), \qquad \textrm{if } -\frac{1}{2} < \re \alpha \leq 0;
\end{equation}

\item[(d)] $L^{(2)}(\eta)$ has the following behavior at infinity:
\begin{equation}
L^{(2)}(\eta) = O\left(\frac{1}{\eta}\right), \quad \eta \to \infty.
\end{equation}
\end{itemize}

The above RH problem for $L^{(2)}$ is very similar to that for $\Psi^{(2)}_\a$ in \eqref{Psi2-J-neg}-\eqref{Psi2-origin}, where only the factor $e^{\theta(\l)\sigma_3}$ in \eqref{Psi2-J-neg}-\eqref{Psi2-J-pos} replaced by $e^{-i(\eta-\pi\nu)\sigma_3}$. Then, by noting the fact that $(e^{\pm i\eta \sigma_3})^{*}=e^{\mp i\eta \sigma_3}$ for $\eta \in\mathbb{R}$ and $(e^{\pm \pi i \nu \sigma_3})^{*}=e^{\pm \pi i \nu \sigma_3}$ since $\nu$ is purely imaginary, a similar analysis gives us $L^{(2)}(\eta) \equiv 0$, which yields $L^{(1)}(\eta) \equiv 0$.

This completes the proof of the vanishing lemma.
\end{proof}

The above vanishing lemma yields the solvability of the RH problem for $L(\eta)$. This is a standard approach to ensure the existence of the solution for a RH problem; for more details, one may refer to  \cite{Dei:Kri:McL:Ven:Zhou1999-2,Fokas1992} as well as \cite[Proposition 2.5]{Claeys2008}.

To construct the solution $P^{0}(z)$ in terms of $L(\eta)$, let us first introduce a conformal mapping in the neighbourhood of $z= 0$ as follows:
\begin{equation}\label{map-z-eta}
\eta(z):= i \tilde{\theta}(z) = z - \frac{4}{3} z^3,
\end{equation}
where $\tilde{\theta}(z)$ is defined in \eqref{tilde-theta-def}. Then, the parametrix $P^{(0)}(z)$ near $z=0$ can be defined as
\begin{equation} \label{p0z-defn}
P^{(0)}(z)=E(z)L(t \eta(z)) e^{-it \eta(z) \sigma_3} \begin{cases}\D
 e^{-i \pi \nu \sigma_3}, & \im z >0,\\
 e^{i \pi \nu \sigma_3}, & \im z <0,
 \end{cases}
\end{equation}
where $E(z)$ is an analytic function in the neighbourhood of 0 given below
\begin{equation}
E(z):=\left(\frac{z+1/2}{1/2 -z}\right)^{\nu \sigma_3}.
\end{equation}
Here the branch cut is chosen such that $\arg(z+1/2)\in(-\pi, \pi)$ and $\arg(1/2-z)\in (-\pi, \pi)$, which implies the following relation:
 \begin{equation}
 E(z) = \left(\frac{z+1/2}{z-1/2}\right)^{\nu \sigma_3}
 \begin{cases}\D
 e^{i \pi \nu \sigma_3}, & \im z >0,\\
 e^{-i \pi \nu \sigma_3}, & \im z <0.
 \end{cases}
 \end{equation}
Similar to Section \ref{sec:local-z+}, from the RH problem for $L(\eta)$, it is easy to verify that $P^{(0)}(z)$ defined in \eqref{p0z-defn} indeed satisfies all the conditions in the RH problem listed at the beginning of the current section. This completes the parametrix construction in the neighbourhood of $z=0$.

\subsection{Final transformation}

Now we proceed to the final transformation by defining
\begin{equation}\label{error}
R(z)=\begin{cases}\D T(z)(P^{(r)})^{-1}(z), & \quad z \in U(z_+,\delta)\setminus \Sigma_T;\\
T(z)(P^{(l)})^{-1}(z), & \quad z \in U(z_-,\delta)\setminus \Sigma_T; \\
T(z)(P^{(0)})^{-1}(z), & \quad z \in U(0,\delta)\setminus \Sigma_T; \\
T(z) N^{-1}(z), &\quad elsewhere.
\end{cases}
\end{equation}
Then, the error function $R(z)$ solves the following RH problem:
\begin{itemize}
\item[(a)] $R(z)$ is analytic for $ z \in \mathbb{C}\cut \Sigma_R$; see Figure \ref{contour(error)}.
\begin{figure}[h]
  \centering
  \includegraphics[width=12cm]{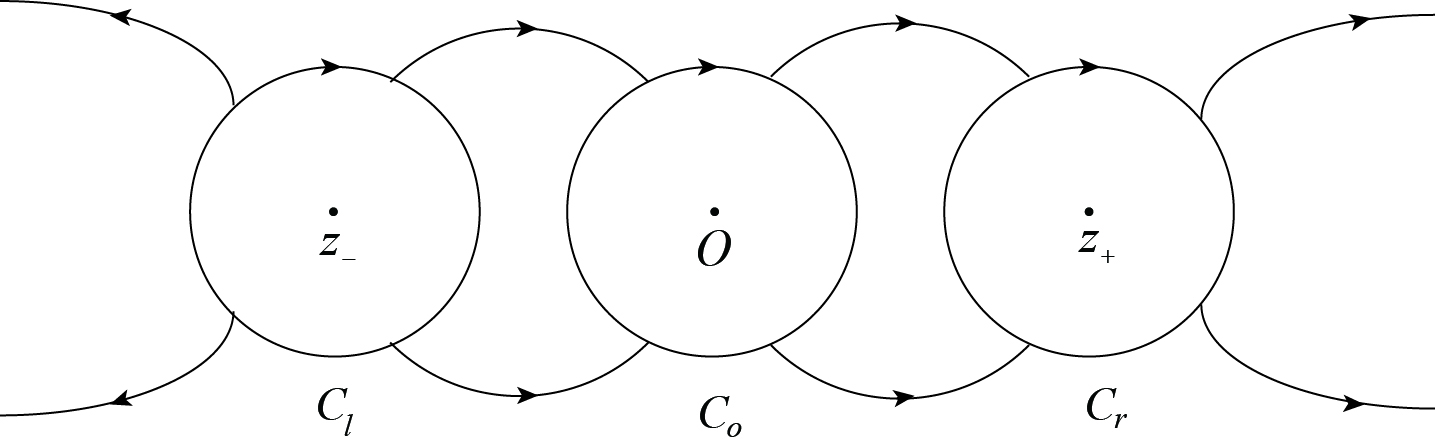}
  \caption{The contour $\Sigma_R$ for the RH problem of $R(z)$. Here $C_l$, $C_r$ and $C_0$ denote the circles centered $z_\pm$ and 0 with small radius $\delta$, oriented in the clockwise direction;}\label{contour(error)}
\end{figure}

\item[(b)] $R(z)$ satisfies the following jump conditions
\begin{equation}
R_{+}(z)=R_{-}(z) J_{R}(z), \quad z\in \Sigma_R,
\end{equation}
where
\begin{equation}\label{Jr-def}
J_{R}(z)=\begin{cases}\D P^{(r)}(z) N^{-1}(z),& \quad z\in \partial {U}(z_{+}, \delta),\\
P^{(l)}(z) N ^{-1}(z),& \quad z\in \partial {U}(z_{-}, \delta),\\
P^{(0)}(z) N^{-1}(z),& \quad z\in \partial {U}(0, \delta),\\
N(z)J_{T}(z) N^{-1}(z),& \quad z\in \Sigma_R \cut \{ \partial {U}(z_{\pm}, \delta) \cup \partial {U}(0, \delta) \};
\end{cases}
\end{equation}
\item[(c)]At infinity, the asymptotic behavior of $R(z)$ is as follows:
\begin{equation}
R(z)= I+O(1/z), \quad z\to \infty.
\end{equation}
\end{itemize}

Due to the matching conditions \eqref{psir-psid-match} and \eqref{psi0-psid-match} and the definition of the jump matrices $J_T(z)$ in Figure \ref{contour(LD)}, it follows that, as $ t \to \infty$,
\begin{equation}\label{Jr-t-order}
J_{R}(z)=\begin{cases}\D
I+O(t^{-\frac{1}{2}}), & \quad z \in \partial U(z_{\pm}, \delta),\\
I+O(t^{-1}), & \quad z \in \partial U(0, \delta),\\
I+O(e^{-ct}), & \quad z\in \Sigma_R \cut \partial {U}(z_{\pm}, \delta) \cup \partial {U}(0, \delta),
\end{cases}
\end{equation}
where $c$ is a positive constant. Moreover, the error terms in the above formula hold uniformly for $z$ in the corresponding contours.
Therefore, we have
\begin{equation}\label{est-Jr}
\|J_{R}(z)-I\|_{L^{2}\cap L^{\infty}(\Sigma_R)} = O(t^{-\frac{1}{2}}).
\end{equation}
It is well-known that the RH problem for $R(z)$ is equivalent to an integral equation as follows:
\begin{equation}\label{R-Cau-int}
R(z) = I +\frac{1}{2\pi i} \int_{\Sigma_R} \frac{R_{-}(z')(J_{R}(z')-I)}{z' - z} dz'.
\end{equation}
Based on the standard procedure of norm estimation of the above Cauchy operator, it follows that for sufficiently large $t$ the relevant integral operator is contracting, and the integral equation can be solved in $L^2(\Sigma_R)$ by iterations; see the standard arguments in \cite{Deift1999,Dei:Kri:McL:Ven:Zhou1999-2}. Then, we have, as $t \to \infty$,
\begin{equation}\label{est-R}
R(z) =I + O(t^{-\frac{1}{2}}), \qquad \textrm{uniformly for } z \in \mathbb{C} \setminus \Sigma_R.
\end{equation}
This completes the nonlinear steepest descent analysis.

\section{Proof of the main results} \label{sec-main-proof}

We first show the following asymptotic result of $u(x;\a)$ for general Stokes multipliers $s_1$ and $s_3$ (where $s_2$ is assumed to be 0 at the beginning of Section \ref{sec:thm2-bc:thm3}).  Then, by choosing the specific Stokes multipliers in \eqref{stokes-real-im} together with the conditions \eqref{real-alpha-k} and \eqref{im-alpha-k} for the real and purely imaginary Ablowitz-Segur solutions, we will prove our Theorems \ref{main-thm} and \ref{main-thm-im}.

\begin{lemma} \label{lemma-u-asy}
  With the Stokes multipliers in \eqref{stokes-real-im}, we have the following asymptotic result for $u(x;\a) $ associated with the RH problem for $\Psi_\a$ satisfying the conditions \eqref{Psi+Psi-}-\eqref{Psi-origin-asy-alpha-neg}:
  \begin{eqnarray}\label{u-conjugate}
u(x;\a) &=& (-x)^{-1/4}\left\{ \frac{\sqrt{\pi} e^{-i\frac{\pi \nu}{2}}}{s_1 \Gamma(\nu)} e^{i\frac{2}{3}(-x)^{3/2}+\nu \ln {8(-x)^{3/2}} -i\frac{\pi}{4}}\right.\nonumber\\
 && \qquad \qquad \qquad \left. + \frac{\sqrt{\pi} e^{-i\frac{\pi \nu}{2}}}{s_3 \Gamma(-\nu)} e^{-i\frac{2}{3}(-x)^{3/2}-\nu \ln {8(-x)^{3/2}}+i\frac{\pi}{4}}\right\} +O(|x|^{-1}),
\end{eqnarray}
 $x \to -\infty$, where 
\begin{eqnarray}
\nu &=& -\frac{1}{2\pi i}\ln (1-s_1 s_3) \nonumber\\
&=&\begin{cases}\D
-\frac{1}{2\pi i}\ln (\cos^2 {\pi \a}-k^2) \quad  \textrm{for} \quad \a\in(-\frac{1}{2}, \frac{1}{2}),\quad k\in (-\cos {\pi \a}, \cos {\pi \a}), \vspace{10pt}\\
\D-\frac{1}{2\pi i}\ln (\cosh^2 {\pi i \a}+|k|^2) \quad  \textrm{for} \quad \a\in i\mathbb{R},\quad k\in i\mathbb{R}.
\end{cases}
\end{eqnarray}
\end{lemma}
\begin{proof}
  To derive the asymptotics of the PII solution  $u(x;\alpha)$, we trace back the transformation $R \to T \to U \to \Psi_\a$ and make use of the formulas \eqref{Psi-U}, \eqref{def-T} and \eqref{error}. Recalling the relation that $u(x;\alpha)=2\biggl(\Psi_{-1}(x)\biggl)_{12} $, we obtain the following expression for $u(x;\alpha)$ from the integral representation of $R(z)$ in \eqref{R-Cau-int}:
\begin{equation}
u(x;\alpha)=2\sqrt{-x}\lim_{z\to\infty}(zR_{12}(z)) = -\frac{\sqrt{-x}}{\pi i} \int_{\Sigma_R} \biggl(R_{-}(z')(J_R(z')-I) \biggr)_{12} dz'.
\end{equation}
Combining \eqref{Jr-t-order} and \eqref{est-R}, one can see that the leading contribution of the above integral is $O(t^{-1/2})$, which comes from $\partial U(z_{\pm}, \delta)$. And the contributions from $\partial U(z_{0}, \delta)$ and $\Sigma_R \cut \partial {U}(z_{\pm}, \delta) \cup \partial {U}(0, \delta)$ are relatively smaller, which are of order $O(t^{-1})$ and $O(e^{-ct})$, respectively. Let us use $C_l$ and $C_r$ to denote the circles $\partial U(z_{\pm}, \delta)$ (see Figure \ref{contour(error)}), then the above formula gives us
\begin{equation} \label{ux-integral-1}
u(x;\alpha)=-\frac{\sqrt{-x}}{\pi i} \left[ \int_{C_r \cup C_l} (J_R)_{12}(z') dz' + O(t^{-1}) \right], \quad \textrm{as } t \to \infty.
\end{equation}
Therefore, more detailed asymptotic information of $J_R(z)$ for $z \in C_r\cup C_l$ is required.
Recalling the asymptotics of $P^{(r)}(z)$ in \eqref{prz-large-t}, the relation between $P^{(r)}(z)$ and $ P^{(l)}(z)$ in \eqref{prl-realtion}, and the fact that $J_R(z)=P^{(r,l)}(z) N^{-1}(z)$ for $z \in C_r\cup C_l$, we have, as $t \to \infty$,
\begin{equation}\label{asy-Jr}
J_R(z)= \begin{cases}\D
\left(\begin{matrix}
 1 & -\frac{\nu s_3}{h_1}e^{\frac{2it}{3}} \beta^2(z)\frac{1}{\sqrt{t}\zeta(z)}\\
-\frac{h_1}{s_3}e^{-\frac{2it}{3}} \beta^{-2}(z)\frac{1}{\sqrt{t}\zeta(z)} & 1
\end{matrix} \right)(I+O(t^{-1})), & \quad z\in C_r,\\
\left(\begin{matrix}
 1 & \frac{h_1}{s_3}e^{-\frac{2it}{3}} \beta^{-2}(-z)\frac{1}{\sqrt{t}\zeta(-z)}  \\
\frac{\nu s_3}{h_1}e^{\frac{2it}{3}} \beta^2(-z)\frac{1}{\sqrt{t}\zeta(-z)} & 1
\end{matrix} \right)(I+O(t^{-1})), & \quad z\in C_l.
\end{cases}
\end{equation}
Note that, the circles $C_l$ and $C_r$ are symmetric about the origin, and one may make a transformation $z \to -z$ such that the integrals in \eqref{ux-integral-1} involve the contour $C_r$ only. Then, combining the above two formulas, we get
\begin{equation}\label{contour-int}
 u(x;\alpha) = \frac{\sqrt{-x}}{\pi i\sqrt{t}} \left [ \frac{\nu s_3}{h_1}e^{2it/3} \int_{C_r} \frac{\beta^2(z')}{\zeta(z')} dz' +  \frac{h_1}{s_3}e^{-2it/3}\int_{C_r} \frac{\beta^{-2}(z')}{\zeta(z')} dz' +O(t^{-1}) \right],
\end{equation}
as $t \to \infty$.

From the definitions of $\zeta(z)$ and $\beta(z)$ in \eqref{zeta-defn} and \eqref{beta-defn}, one can see that $\zeta(z)$, $\beta(z)$ and $\beta^{-1}(z)$ are analytic in the neighbourhood of $z_+ = 1/2$. Moreover, $\zeta(z)$ has a simple zero at $1/2$; see \eqref{zeta-z+-nei}.
By Cauchy's residue theorem, we have from \eqref{zeta-z+-nei} and \eqref{beta-defn}
\begin{equation}
\int_{C_r} \frac{\beta^{\pm 2}(z')}{\zeta(z')} dz' =-2\pi i\Res_{z=\frac{1}{2}}\left(\frac{\beta^{\pm 2}(z)}{\zeta(z)}\right)=\frac{\pi i}{\sqrt{2}}e^{\frac{\pi i}{4}\pm \frac{3\pi i \nu}{2}}(8t)^{\pm \nu}.
\end{equation}
Since $\re\nu=0$, the above quantity is bounded as $t \to \infty$. Then, the above formula and \eqref{contour-int} give us
\begin{equation}
u(x;\alpha)=\frac{e^{\pi i/4}}{\sqrt{2} (-x)^{1/4}}\left(\frac{\nu s_3}{h_1}e^{\frac{2it}{3}+\frac{3\pi i \nu}{2}}(8t)^{\nu}+\frac{h_1}{s_3}e^{-\frac{2it}{3}-\frac{3\pi i \nu}{2}}(8t)^{-\nu}\right) +O(|x|^{-1}),
\end{equation}
as $x\to -\infty$ (note that $t=(-x)^{3/2}$). Finally, the formula \eqref{u-conjugate} follows from the definition of $h_1$ in \eqref{def-h-01},  the following identity for Gamma function
\begin{equation}\label{Gamma-identity}
\Gamma(\nu)\Gamma(-\nu)=-\frac{\pi}{\nu \sin \pi \nu},
\end{equation}
and $\nu=-\frac{1}{2\pi i}\ln (1-s_1 s_3)$.

This completes the proof of our lemma.
\end{proof}

With the above lemma, we are ready to prove our Theorems \ref{main-thm} and \ref{main-thm-im}.

\bigskip

\noindent\emph{Proof of Theorem \ref{main-thm}.} Note that, when $s_3=\bar{s}_1$, we have from \eqref{u-conjugate}
\begin{equation}\label{u-Re}
u(x;\a)= 2(-x)^{-1/4}\re \left\{ \frac{\sqrt{\pi} e^{-i\frac{\pi \nu}{2}}}{s_1 \Gamma(\nu)} e^{i\frac{2}{3}(-x)^{3/2}+\nu \ln {8(-x)^{3/2}} -i\frac{\pi}{4}}\right\} +O(|x|^{-1}),
\end{equation}
as $x\to -\infty$. Recalling the Stokes multipliers for the real Ablowitz-Segur solution $u_{\textrm{AS}}(x;\a)$ in \eqref{stokes-real-im} with the condition \eqref{real-alpha-k}, we have
\begin{equation}
  \nu =-\frac{1}{2\pi i}\ln (1-s_1s_3)= -\frac{1}{2\pi i}\ln (1-|s_1|^2)  = -\frac{1}{2\pi i} \ln\biggl( \cos^2(\pi \a) - k^2\biggr).
\end{equation}
From the properties of Gamma functions, we get
\begin{equation}\label{Gamma-mod}
\left|\frac{1}{\Gamma(\nu)} \right|^2 =\frac{i\nu}{2\pi}e^{\pi i \nu}|s_1|^2.
\end{equation}
Combining the above three formulas, we obtain the asymptotic expansion of $u_{\textrm{AS}}(x;\a)$  in \eqref{asy-neg} as well as the connection formulas \eqref{d-k} and \eqref{phi-k}.

The pole-free property of $u_{\textrm{AS}}(x;\a)$ follows from Lemma \ref{lemma-vanishing-1}, which is the vanishing lemma for the RH problem of $\Psi_\a$ with the Stokes multipliers in \eqref{stokes-real-im} under the condition \eqref{real-alpha-k}.

This finishes the proof of Theorem \ref{main-thm}.\hfill $\Box$

\bigskip

Similarly, we prove  Theorem \ref{main-thm-im} below.

\bigskip

\noindent\emph{Proof of Theorem \ref{main-thm-im}.} When $ s_3 = -\bar{s}_1$, the formula \eqref{u-conjugate} gives us
\begin{equation}\label{u-im}
u(x;\a)= 2i(-x)^{-1/4}\im \left\{ \frac{\sqrt{\pi} e^{-i\frac{\pi \nu}{2}}}{s_1 \Gamma(\nu)} e^{i\frac{2}{3}(-x)^{3/2}+\nu \ln {8(-x)^{3/2}} -i\frac{\pi}{4}}\right\} +O(|x|^{-1}),
\end{equation}
as $x\to -\infty$. With the specific Stokes multipliers in \eqref{stokes-real-im} and the condition \eqref{im-alpha-k} for the purely imaginary Ablowitz-Segur solution $u_{iAS}(x;\a)$, one can also see that $\nu$ is a purely imaginary number, namely,
\begin{equation}
  \nu =-\frac{1}{2\pi i}\ln (1-s_1s_3) = -\frac{1}{2\pi i} \ln\biggl( \cosh^2(\pi i \a) + |k|^2\biggr).
\end{equation}
Similar to \eqref{Gamma-mod}, we have
\begin{equation}
\left|\frac{1}{\Gamma(\nu)} \right|^2  = \frac{-i\nu}{2\pi}e^{-\pi i \nu}\frac{|s_1|^2}{1+|s_1|^2}.
\end{equation}
Thus, from the above three formulas, we arrive at the asymptotic expansion of $u_{i\textrm{AS}}(x;\a)$  in \eqref{asy-neg-im} as well as the connection formulas \eqref{d-k-im} and \eqref{phi-k-im}.

The pole-free property of $u_{i\textrm{AS}}(x;\a)$ is obvious, since the residues of all poles of PII transcendents are 1 or $-1$.

This finishes the proof of Theorem \ref{main-thm-im}.\hfill $\Box$

\section*{Acknowledgements}

We would like to thank the referees for their helpful suggestions and comments.

The authors were partially supported by grants from the Research Grants Council of the Hong
Kong Special Administrative Region, China (Project No. CityU 11300814, CityU 11300115, CityU 11303016).

\end{document}